\documentclass[11pt,oneside,a4paper]{article}

\usepackage[top=20mm, right=20mm, bottom=20mm, left=20mm]{geometry}
\usepackage{amsfonts, amsbsy, amsmath, amsthm, amssymb, latexsym}
\usepackage{mathrsfs}
\usepackage{color}
\usepackage{lscape}

\newtheorem{theorem}{Theorem}[section]
\newtheorem{lemma}[theorem]{Lemma}
\newtheorem{propn}[theorem]{Proposition}
\newtheorem{corol}[theorem]{Corollary}

\theoremstyle{definition}

\newtheorem{definition}[theorem]{Definition}

\newtheorem*{lemma*}{Lemma}

\newcommand{\SM}{\setminus\{1\}}
\newcommand{\C}{\mathscr{C}}
\newcommand{\Chat}{\widehat{\mathscr{C}}}

\title{ On Centres of $3$-Blocks of the Ree groups $^2G_2(q)$}
\author{Julian Brough$^1$ and Inga Schwabrow$^2$}

\begin{document}
\date{}
\maketitle

\vspace{-10mm}
\begin{center}

\small
\textit{$^1$ FB Mathematik, TU Kaiserslautern, Postfach 3049, 67653 Kaiserslautern, Germany}

\text{E-mail: brough@mathematik.uni-kl.de}

\textit{$^2$ School of Mathematics, The University of Manchester, Manchester, M13 9PL, UK}

\text{E-mail: inga.schwabrow@gmx.de}
\end{center}

\normalsize

\begin{abstract}
Let $G:={^2G_2}(q)$ be the simple Ree group with $q=3^{2k+1}$ and $k$ a positive integer. We show that the centre of the principal block $Z(kGe_0)$, where $k$ is an algebraically closed field of characteristic $3$, is not isomorphic to the centre of the Brauer corresponding block $Z(kN_G(P))$, where $N_G(P)$ is the normaliser in $G$ of a Sylow $3$-subgroup.
As part of the proof, we compute the conjugacy classes of elements and the character tables of the maximal parabolic subgroups of $G$.
\end{abstract}

\section{Introduction}

Brou{\'e}'s conjecture postulates that if a block has abelian defect groups, then there exists a perfect isometry between the block and its Brauer correspondent. 
If such a perfect isometry exists, it follows that the centres of the blocks are isomorphic over a sufficiently large complete discrete evaluation ring $\mathcal{O}$. 
In particular,  $Z(\mathcal{O}G\hat{e_0})\cong Z(\mathcal{O}N_G(P)\hat{f_0})$, for $\hat{e_0}$ and $\hat{f_0}$ the corresponding principle block idempotents.

Originally, Brou{\'e} also made this conjecture for the principal blocks, in the case that the Sylow normaliser controls the $p$-fusion \cite{Broue88}. An initial counterexample to this conjecture was given by
 Brou{\'e} and Serre \cite[Section 6]{Broue1}. In particular, they considered the Cartan matrices to conclude that there is no derived equivalence in the case of $^2B_2(8)$ with $p=2$. 

Later on, Cliff elaborated on this counterexample by considering the radical structure of the centres of the blocks \cite{Cliff}.
Let $G$ be a finite simple group of Lie type $^2B_2(q)$ and $P$ a Sylow $2$-subgroup $G$, where $q=2^{2k+1}\geq 8$. 
Cliff observed that over a field $k$ of characteristic $2$, $Z(kGe_0)\cong Z(kN_G(P))$, while over a discrete valuation ring $\mathcal{O}$ of characteristic zero $Z(\mathcal{O}G\hat{e_0})\not\cong Z(\mathcal{O}N_G(P))$ \cite[Theorem 4.1]{Cliff}.
This provided an infinite family of pairs of blocks whose centres are isomorphic over a field $k$ of characteristic $p$, but are not perfectly isometric over a discrete valuation ring of characteristic $0$.

In this paper we extend the ideas of Cliff to the simple groups of Lie type $^2G_2(q)$, where $q=3^{2k+1}$, and show that in this case there does not exist an isomorphism of centres over a field $k$ of characteristic $3$, which implies there does not exist one over $\mathcal{O}$.
To do this we analyse the radical structure of the centres of the principle blocks of $^2G_2(q)$ and $N_G(P)$ for $P$ a Sylow $3$-subgroup. 
Note that the group algebra $kN_G(P)$ is indecomposable.
Although $Z(kGe_0)$ and  $Z(kN_G(P))$ have the same dimension over $k$, we will show in Theorem \ref{th:MainReeLL3} that $LL(Z(kGe_0))=3$, whereas Theorem \ref{th:Snowman} states that $LL(Z(kN_G(P)))=2$; hence no such isomorphism can exist.

For a Sylow $3$-subgroup $P$ of $^2G_2(q)$, the normaliser is isomorphic to $P\rtimes C_{q-1}$ \cite{Ward}.
Therefore $N_G(P)$ is a solvable group; furthermore $O_{3'}(N_G(P))=1$ and so $N_G(P)$ has a unique $3$-block $b$ \cite[Proposition III.1.12]{KarpilovskyJacRadical}.
This implies that the principle block of $^2G_2(q)$, $B_0$, is the unique $3$-block of maximal defect.
Moreover as $^2G_2(q)$ has trivial intersection Sylow $3$-subgroups \cite[p. 307]{CFSG}, its blocks either have maximal defect or defect zero.
However $^2G_2(q)$ only has one character of $3$-defect zero (the Steinberg character), thus the group $^2G_2(q)$ has two  blocks.
The group $^2G_2(q)$ has $q+8$ irreducible characters \cite[p. 85]{Ward}, from which it now follows by \cite[Proposition 6.2]{Blau} that $k(B_0)=q+7=k(b)$, where $k(B)$ is the number of irreducible complex characters in a block $B$. 

In order to study the Loewy length of the center of these group algebras, we make use of the character tables.
The (complex) irreducible characters of $^2G_2(q)$ were computed by Ward \cite{Ward}. 
For the normaliser of a Sylow $3$-subgroup, we use the character degrees obtained by Eaton \cite{Eaton1} and a partial segment of the character table computed by Gramain \cite{GramainPhD} to produce the full character table,  Table~\ref{tb:CharacterTabReeNorm}.
Note that a complete character table for $N_G(P)$ was given by Landrock and Michler \cite[p. 88]{LandrockMichler}. 
However as only vague details  were provided, we give an independent and much more detailed construction of the character table here.

\vspace{4mm}
\noindent
We summarise the main result obtained in this paper.
\begin{theorem}\label{th:FullStatmentRee} Let $k$ be an algebraically closed field of characteristic $3$ and $G= {^2G_2}(q)$, $q=3^{2k+1}\geq 27$, $P\in {\rm Syl}_3(G)$. Then  $LL(Z(kGe_0))=3> 2 =LL(Z(kN_G(P)))$; hence
\[Z(kGe_0)\not\cong Z(kN_G(P)).\]
\end{theorem}

As an immediate corollary, we get the following result.
\begin{corol} There is no perfect isometry, and hence no derived equivalence, between $kGe_0$ and $kN_G(P)$.
\end{corol}

\section{Preliminary}

Let $(K,\mathcal{O},k)$ be a $p$-modular system; that is,  $K$ is a field of characteristic zero, $\mathcal{O}$ is a complete valuation ring with unique maximal ideal $ J(\mathcal{O})$, and $k$ is a field of characteristic $p$. In addition $\mathcal{O}/J(\mathcal{O})  \cong k$ and $K$ is the field of fractions of $\mathcal{O}$.

\subsection{The class algebra constants}

The conjugacy class sums form a basis of $Z(kG)$, and therefore the product of any conjugacy class sums must be a sum of conjugacy class sums. 
Let $\C(x)$ denote the conjugacy class of $x$ and 
\[
\Chat(x):=\sum\limits_{g\in \C(x)}g
\]
the class sum of $\C(x)$ which is an element in $\mathcal{O}G$.
Then the following common notation is adopted:
\begin{equation}\label{eq:Dog} \Chat(x)\Chat(y)=\sum_{z\in \mathscr{P}}a(x,y,z)\Chat(z) \;\;\;\;\;\text{ for } x,y\in G,\end{equation}
where $\mathscr{P}$ is a set of representatives of the conjugacy classes of $G$ and the constants $a(x,y,z)$ are referred to as the class algebra constants.
We note that from the definition, it follows that the structure constants $a(x,y,z)$ lie in $\mathbb{Z}$.

Burnside's original work in representation theory over $\mathbb{C}$ provided a method for obtaining the class algebra constants from the character table of the group. In particular, this connection is made precise by Burnside's formula \cite[p.316]{Burnside} which forms a crucial part in the study of representation theory and will play a large role in the calculations to follow in this paper.

Given $x,y,z\in G$, Burnside's formula states:
\begin{equation}\label{eq:StructureConstants}
a(x,y,z)=\frac{|G|}{|C_G(x)||C_G(y)|}\sum_{\theta\in {\rm Irr}(G)} \frac{\theta(x)\theta(y)\theta(z^{-1})}{\theta(1)}.
\end{equation}
From this formula it is clear that $a(x,y,z)=a(y,x,z)$.
Additionally we make the following observation, which reduces the number of explicit calculations required to compute all the structure constants.

\begin{lemma}\label{lm:inverseclass} Given $x,y,z \in G$, then $ a(x,y,z)=  a(x^{-1}, y^{-1}, z^{-1})$.
Furthermore 
\[
a(x,y,z)=\frac{|C_G(z)|}{|C_G(y)|}a(z^{-1},y,x^{-1})=\frac{|C_G(z)|}{|C_G(y)|}a(x^{-1},z,y).
\]
\begin{proof}
As $xy=z$ if and only if $y^{-1}x^{-1}=z^{-1}$, it follows that $a(x,y,z)=a(y^{-1},x^{-1},z^{-1})=a(x^{-1},y^{-1},z^{-1})$.
We note that the second statement follows by swapping either $\theta(x)$ with $\theta(z)$ or $\theta(y)$ with $\theta(z)$ and taking the complex conjugate in Burnside's formula.
\end{proof}
\end{lemma}

As we are working over a field of  positive characteristic, we are interested when the field characteristic divides the structure constant $a(x,y,z)$.
Using the notion of the defect of a conjugacy class provides one method to determine this in certain cases. 

\begin{definition}
Let $g$ be an element of a finite group $G$.
Then the $p$-defect of the conjugacy class of $g$ in $G$ is given by $d_g$, where $p^{d_g}$ is the order of a Sylow $p$-subgroup of $C_G(g)$.
\end{definition}

The following corollary follows from Lemma~\ref{lm:inverseclass}.

\begin{corol}\label{th:defectConClass} \cite[Cor 87.7]{CurtisReiner} Let $k$ be a field of characteristic $p$, and for $g\in G$, let $d_g$ be the defect of the conjugacy class $\C(g)$. If $d_y< d_z$ or $d_x<d_z$ then $p\;|\; a(x,y,z)$. Therefore
\[\Chat(x)\Chat(y)=\sum_{d_z\leq {\rm min}\{d_x, d_y\}}a(x,y,z)\Chat(z) \;\;\in kG.\]
\end{corol}

\begin{propn}\label{th:sumStrucConst} Fix $y,z \in G$. Then
\[ \sum_{x\in \mathscr{P}} a(x,y,z)= |\C(y)|.\]
\end{propn}
\begin{proof}
We have
\[ \sum_{x\in \mathscr{P}} \Chat(x)\Chat(y)=\sum_{g\in G} g\Chat(y)=\sum_{y' \in \C(y)} \hat{G} y'= |\C(y)|\;\sum_{g\in G} g=|\C(y)|\;\sum_{z\in \mathscr{P}} \Chat(z). \]
On the other hand
\[ \sum_{x\in \mathscr{P}} \Chat(x)\Chat(y)=\sum_{x\in \mathscr{P}}\sum_{z\in \mathscr{P}} a(x,y,z)\Chat(z) . \]
Therefore the coefficient of $\Chat(z)$ is $\sum_{x\in \mathscr{P}} a(x,y,z)$ and the proposition follows.
\end{proof}

We end this section with a result which relates the structure constants of a group $G$, with trivial intersection Sylow $p$-subgroups, to the normaliser of $Q$ a Sylow $p$-subgroup.
In particular, we generalise the result of Cliff for the Suzuki groups \cite[Lemma 3.2]{Cliff} by observing that the argument holds whenever the Sylow $p$-subgroups have the trivial intersection property.
Note that for such a group $G$, we have that $N_G(Q)$ controls fusion of $Q$ in $G$; that is, if $x^g\in Q$ for $x\in Q$ and $g\in G$, then either $x=1$ or $g\in N_G(Q)$. 
This follows from the observation that $x^g\in Q\cap Q^g$ which equals $Q$ or is trivial.

\begin{propn}\label{prop:a(xyz)generic} Let $G$ be a finite group with trivial intersection Sylow $p$-subgroups.
Let $x,y,z\in \; Q\setminus\{1_G\}$ where $Q\in  {\rm Syl}_p(G)$. Then $a(x,y,z)\equiv a_H(x,y,z)\;\;{\rm mod} \;|C_G(z)|_p$ where $H= N_G(Q)$.
\end{propn}
\begin{proof}
Let \[\mathcal{A}=\{ (x',y')\;|\;x'\in x^G, y'\in y^G, x'y'=z\}.\]  Then $|\mathcal{A}|=a(x,y,z)$ and $\mathcal{A}$ can be split into two disjoint sets $\mathcal{A}_1\cup \mathcal{A}_2$, where
\[\begin{array}{rcl}
\mathcal{A}_1&=& \{ (x',y')\;|\;x'\in x^G\cap Q,\; y'\in y^G, \;x'y'=z\}  \\
\mathcal{A}_2&=& \{ (x',y')\;|\;x'\in x^G\setminus Q,\; y'\in y^G,\; x'y'=z\} \\
\end{array}\]
Note that as $z=x'y'\in Q$, we have $x'\in x^G\cap Q$ if and only if  $y'\in y^G\cap Q$.

\vspace{5mm} \noindent
\textbf{Claim 1} $x^G\cap Q=x^{N_G(Q)}$

Suppose $x' \in x^G\cap Q$. Then there exists an element $g\in G$ such that $x^g=x' \in Q\cap x^G$. Since the Sylow $p$-subgroups of $G$ are trivial intersection, by the above remark, there exists $h\in N_G(Q)$ such that $x^g=x^h$. Hence $ x'=x^g=x^h\in x^{N_G(Q)}$.

Therefore $|\mathcal{A}_1|=a_H(x,y,z)$.

\vspace{4mm} \noindent
\textbf{Claim 2} The size of $\mathcal{A}_2$ is divisible by $|C_G(z)|_p$.

Take an element $z$ in $Q\SM$. Suppose $g\in  N_G(\langle z\rangle)$, so $ \langle z\rangle^g=\langle z\rangle$; this implies $\langle z\rangle^g \in Q\cap Q^g=\{1\}$, a contradiction unless $g\in N_G(Q)$. Hence the normaliser and thus the centraliser of $z$ is a subgroup of $N_G(Q)$, i.e.  $C_G(z)\leq N_G(\langle z\rangle)\leq N_G(Q).$

\vspace{4mm}
Note that if $g\in C_G(z)$ then $z^g=z=(x'y')^g=x'^gy'^g$; since $C_G(z)\leq N_G(Q)$, if $(x')^g\in Q$ then $x'\in Q^g=Q$.  Therefore $C_G(z)$ acts on $\mathcal{A}_2$.

\vspace{4mm}
Suppose $(x',y')\in \mathcal{A}_2$. By the orbit stabiliser theorem,  the size of a $C_G(z)$-orbit containing $(x',y')$ is given by
\[\begin{array}{ccl}
| (x',y')^{C_G(z)}| &=& [C_G(z)\;:\;(C_{C_G(z)}(x')\cap C_{C_G(z)}(y'))]\\
&=& [C_G(z)\;:\;(C_G(z)\cap C_G(x')\cap C_G(z)\cap C_G(y'))]\\
\end{array}\]
Consider $C_G(z)\cap C_G(x')$, and suppose $x'=x^g$ such that $g\not\in N_G(Q)$; then $C_G(z)\cap C_G(x^g)\leq N_G(Q)\cap N_G(Q^g)$. Suppose $S\in {\rm Syl}_p(N_G(Q)\cap N_G(Q^g))$. As $S\leq N_G(Q)$ it follows that  $S\leq Q$; similarly  $S\leq N_G(Q^g)$ implies $S\leq Q^g$. Hence as $g\not\in N_G(Q)$,  $S\leq Q\cap Q^g= 1\SM$, and it follows that $p$ does not divide $|C_G(z)\cap C_G(x')|$. In particular, $p$ does not divide $|C_G(z)\cap C_G(x')\cap C_G(y')|$.

Hence $|(x',y')^{C_G(z)}|_p= |C_G(z)|_p$ for all pairs $(x',y')\in\mathcal{A}_2$ and therefore $|C_G(z)|_p$ divides the size of $\mathcal{A}_2$.

Finally, combining the two claims, $|\mathcal{A}|\equiv |\mathcal{A}_1|  \; {\rm mod}\; |C_G(z)|_p$, and  since $\mathcal{A}_1$ and $\mathcal{A}_2$ are disjoint, $|\mathcal{A}|\equiv |\mathcal{A}_1|\equiv a_H(x,y,z) \; {\rm mod}\;  |C_G(z)|_p$.
\end{proof}

\subsection{The character table of the Sylow normaliser}
In this section we shall construct the full complex character table for the normaliser of a Sylow $3$-subgroup $P$ of $^2G_2(q)$ with $q=3^{2k+1}$.
First we recall the notation used to describe both the Sylow $3$-subgroup and its normaliser.
The standard references for the description presented are \cite{SGLT} and \cite{GramainPhD}.

Let $P$ be a Sylow $3$-subgroup of $G$. 
Then $P$ is isomorphic to the set of elements $\{x(t,u,v)\;|\; t,u,v\in \mathbb{F}_q\}$ endowed with a multiplication given by
\[x(t_1,u_1,v_1)x(t_2,u_2,v_2)= x(t_1+t_2, u_1+u_2-t_1t_2^{3\theta}, v_1+v_2-t_2u_1+t_1t_2^{3\theta+1}-t_1^2t_2^{3\theta})\]
 where $\theta$ is the automorphism of $\mathbb{F}_q$ given by $\lambda^{\theta}=\lambda^{3^k}$ for all $\lambda\in \mathbb{F}_q$.
Observe that the inverse of an element $x(t,u,v)\in P$ is given by
\[x(t,u,v)^{-1}=x(-t,-u-t^{3\theta+1}, -v-tu+t^{3\theta+2}).\]

As $P$ is a normal Sylow $3$-subgroup of $N_G(P)$, there exists a complement $W$ to $P$ in $N_G(P)$.
In particular, $W$ is a cyclic group of order $q-1$ that can be labelled by the set $\{h(w)\;|\;w\in \mathbb{F}_q^{\times}\}$; furthermore the conjugation of $P$ by $W$ is given by
\[h(w)x(t,u,v)h(w)^{-1}=x(w^{2-3\theta}t,w^{3\theta-1}u,wv)\]
where $t,u,v,w \in \mathbb{F}_q, w\neq0$. 
Note that $h(1)=1_W=1_G$ and $h(-1)$ is the unique involution in $W$.

The elements of $P$ form $7$ conjugacy classes in $N_G(P)$ and the class sizes and element orders can be found in \cite[Section 2.5.2]{GramainPhD}.
Since there are $q+7$ conjugacy classes overall in $N_G(P)$, it remains to determine the remaining $q$ conjugacy classes.
The following details are summerised in Table~\ref{tb:ReeNormConCl}.

As $W$ is cyclic, to construct $C_N(h(w))$, it is enough to find $x\in C_P(h(w))$; moreover $\C(h(w))=A_wh(w)$ for some subset $A_w\subset P$.
It is clear that $x(t,u,v)= h(w)x(t,u,v)h(w)^{-1}= x(w^{2-3\theta}t,w^{3\theta-1}u,wv)$ implies $v=0$ and $w^{2-3\theta}=w^{3\theta-1}=1$. 
However $w^{2-3\theta}=1$ has a unique solution $w=1$ in $\mathbb{F}_q$; while $ w^{3\theta-1}=1$ only has two solutions $w=\pm 1$ \cite[p. 62]{GramainPhD}. 
Thus if $w\neq \pm1$, then $C_N(h(w))= W$ and therefore $|A_w|=|P|$ and $\C(h(-1))=Ph(-1)$.
If $w=-1$, then $C_P(h(-1))= \{x(0,u,0)\;|\; u\in \mathbb{F}_q\}$ and so $|C_N(h(-1))|= q(q-1)$. 
Additionally, as $x(t,u,v)h(-1) x(t,u,v)^{-1}=x(t, t^{3\theta+1}, v+tu)h(-1)$, it follows that the $q^2$ elements $x(t, t^{3\theta+1}, v+tu)h(-1)$ form the conjugacy class $\C(h(-1))$. 
We note that $x(0,-1,0)h(-1)$ $\not\in$ $\C(h(-1))$.

Next consider $x(0,-1,0)h(-1)\in N_G(P)$. 
If $x(t,u,v)\in C_P(x(0,-1,0)h(-1))$ then $x(t,u,v)=x(-t,-1+u+w^{3\theta-1},-v-t)$.
In particular, $t=0$ and so $v=0$, which further implies $w^{3\theta-1}=1$.
Hence $C_N(x(0,-1,0)h(-1))=\{x(0,u,0)h(w)\in N_G(P) \:|\; u\in \mathbb{F}_q, w=\pm 1\}$ which has size $2q$.
Note that $(x(0,-1,0)h(-1))^{-1}=h(-1)x(0,1,0)=x(0,1,0)h(-1)$.
However as
\[ 
x(t,u,v)h(w)x(0,-1,0)h(-1) h(w)^{-1} x(t,u,v)^{-1} = x(-t, -w^{3\theta-1}+t^{3\theta+1}, -vt^{3\theta+2}+tw^{3\theta-1})h(-1),
\]
it follows that $(x(0,-1,0)h(-1))^{-1}=x(0,1,0)h(-1)\not\in \C( x(0,-1,0)h(-1))$, otherwise  $t=0$ and $-w^{3\theta-1}=1$, which has no solution $w\in \mathbb{F}_q^{\times}$ as $4$ does not divide $q-1$.
Hence the two remaining conjugacy classes are an inverse pair.

\begin{table}[h]
\centering
\caption{Conjugacy classes of $N_G(P)$}\label{tb:ReeNormConCl}
\[\footnotesize
\begin{array}{c|c|c|c|c|c}
\text{label}  &\C(g)&o(g) & |C_N(g)| &  C_N(g)&|\C(g)|\\\hline
1_N&x(0,0,0)h(1)& 1& |N|&N_G(P)& 1\\\hline
X= Z(P) \setminus\{1_P\}& x(0,0,v)h(1)  &3& q^3&  P&q-1\\
&  v\neq 0&  & & & \\\hline
T, T^{-1}& x(0,u,v)h(1)   &3& 2q^2&   x(0,u,v)h(w)  & q(q-1)/2\\
&  u\neq 0& & &w=\pm1 & \\\hline
Y, YT, YT^{-1}& x(t,u,v)h(1) & 9& 3q& x(0,0,v_2),  & q^2(q-1)/3\\
&  t\neq 0& & & x(t,u,v_2),& \\
 &&&&x(-t,-t^{3\theta+1}-u,v_2) & \\
 &&&&   \text{ where } v_2\in \mathbb{F}_q &  \\\hline
Ph(w)& x(t,u,v)h(w) &|h(w)| & q-1&  x(0,0,0)h(w)  & q^3\\
w\neq \pm 1&  w \text{ fixed },w\neq \pm 1 & & & & \\\hline
J:x(0,0,0)h(-1)& x(t,t^{3\theta+1},v+tu)h(-1) & 2& q(q-1)&  x(0,u,0)h(w) & q^2\\\hline
JT& \alpha &6 & 2q&  x(0,u,0)h(w) & q^2(q-1)/2\\
x(0,-1,0)h(-1)& & & &  w=\pm1 & \\\hline
JT^{-1}& \beta &6 & 2q&  x(0,u,0)h(w) & q^2(q-1)/2\\
(x(0,-1,0)h(-1))^{-1}& & & &  w=\pm1 & \\\hline
\end{array}\normalsize
\]
where $t,u,v \in \mathbb{F}_q$ and $w\in \mathbb{F}_q^{\times}$,
\[
\alpha= x(-t, -w^{3\theta-1 }+t^{3\theta+1},-v-t^{3\theta+1}+tw^{3\theta-1})h(-1)
\] 
and 
\[
\beta=x(-t, w^{3\theta-1 }+t^{3\theta+1},-v-t^{3\theta+1}-tw^{3\theta-1})h(-1).
\]
\end{table}

\subsubsection{Detailed construction of the character table}
In this section we make use of the common notation $\overline{\theta}$ to denote the complex conjugate of the character $\theta$.
Furthermore given two characters $\theta_1$ and $\theta_2$ of $N_G(P)$, then $\langle \theta_1,\theta_2\rangle$ denotes the inner product of the two characters of $N_G(P)$:
\[
\langle \theta_1,\theta_2\rangle:=\frac{1}{N_G(P)}\sum\limits_{g\in N_G(P)}\theta_1(g)\overline{\theta_2(g)}.
\]

Gramain \cite[Section 2.5.4]{GramainPhD} gives the following characters of $N_G(P)$ which are induced from characters in $P$.
\[\begin{array}{ccc}
\text{ {\rm Irr}(P) }& {Ind_P^{N_G(P)}} & \text{ {\rm Irr}}(N_G(P)) \\
\{\lambda_2, \ldots, \lambda_q\}& \rightarrow&\lambda, \text{ degree } q-1\\
\{\psi_2, \ldots, \psi_q\}& \rightarrow&\psi, \text{ degree } (q-1)q\\
\{\chi_{i,j}, i=1,2 \text{ and } 1\leq j \leq \frac{q-1}{2}\}& \rightarrow&\chi, \text{ degree } (q-1)3^k\\
\{\overline{\chi_{i,j}}, i=1,2 \text{ and } 1\leq j \leq \frac{q-1}{2}\}& \rightarrow&\overline{\chi}, \text{ degree } (q-1)3^k\\
\{\chi_{3,j},   1\leq j \leq \frac{q-1}{2}\}& \rightarrow& \mu_1+\mu_2,  \text{ each $\mu_i$ irreducible of degree } \frac{q-1}{2}3^k\\
\{\overline{\chi_{3,j}},   1\leq j \leq \frac{q-1}{2}\}& \rightarrow&\overline{\mu_1}+\overline{\mu_2},  \text{ each $\overline{\mu}$ irreducible of degree } \frac{q-1}{2}3^k\\
\end{array}
\]

In addition to these irreducible characters we obtain $q-1$ linear characters $\alpha_0=1_N, \ldots, \alpha_{q-2}$ for $N_G(P)$ by lifting the characters $\widetilde{\alpha_i}$ with $0\leq i \leq q-2$ from the quotient group $N_G(P)/P\cong C_{q-1}$. 
In other words $\alpha_i(g):=\widetilde{\alpha_i}(gP)$, so $\alpha_i|_P=1$.
Let $W=\langle h\rangle$, and $\widetilde{\alpha_i}$ defined by $\widetilde{\alpha_i}(h^j)=\xi^{ij}$ for $\xi$ a fixed primitive $(q-1)^{th}$ root of unity. 
As the coset $J\cdot P$ is the unique involution in $W$ and $J\cdot P=JT^{\pm1}\cdot P$, it follows that $\alpha_i(JT^{\pm1})=\alpha_i(J)=\widetilde{\alpha_i}(h^{\frac{q-1}{2}})=(-1)^i$.
We now fix the elements $w_j$ such that $Ph(w_j)$ is mapped to $h^j$ in $W$.
For such a labeling we have $\alpha_i(Ph(w_j))=\xi^{ij}$.
Also note that $w_{\frac{q-1}{2}}=-1$ and $Ph(-1)$  labels the union of the conjugacy classes of $J,JT$ and $JT^{-1}$. 
From now on, when we write $Ph(w_j)$ we exclude the case that $w_j=\pm 1$ unless explicitly stated otherwise.

The following table combines the details contained in \cite{Eaton1} (character degrees) and \cite{GramainPhD} (character values) with the above values for the linear characters $\alpha_i$.

\[\footnotesize
\begin{array}{c||ccccccccccccccc}
|\C(g)|& 1&q-1&\frac{q(q-1)}{2}  &\frac{q^2(q-1)}{3}&\frac{q^2(q-1)}{3}& \frac{ q^2(q-1)}{3} &q^3&q^2& \frac{q^2(q-1)}{2}\\[7pt]
|C_N(g)|& q^3(q-1)&q^3&2q^2&3q&3q&3q&q-1&q(q-1)& 2q\\
\\\hline
& 1&X&T,T^{-1}&Y&YT& YT^{-1}&Ph(w_j)&J& JT, JT^{-1}\\\hline\hline
\alpha_0=1_N&1       & 1&1&1& 1&1&1&1&1    \\
\alpha_i&1      &1&1& 1& 1&1& \xi^{ij}& (-1)^i & (-1)^i  \\
\lambda&q-1       & && -1&-1&-1&&&    \\
\mu_1&\frac{3^k(q-1)}{2}       &   & & -\varepsilon3^k& -\varepsilon3^k \overline{\omega}& -\varepsilon 3^k\omega & &&  \\
\mu_2&\frac{3^k(q-1)}{2}       &    & & -\varepsilon3^k& -\varepsilon3^k \overline{\omega}& -\varepsilon 3^k\omega & &&   \\
\overline{\mu_1}&\frac{3^k(q-1)}{2}         &   & & -\varepsilon3^k& -\varepsilon 3^k\omega& -\varepsilon 3^k\overline{\omega}  & &&  \\
\overline{\mu_2}&\frac{3^k(q-1)}{2}         &   & & -\varepsilon3^k& -\varepsilon 3^k\omega& -\varepsilon 3^k\overline{\omega}  & &&   \\
\chi &3^k(q-1)          & &&  \varepsilon3^k&\varepsilon3^k \overline{\omega}& \varepsilon3^k\omega& &&    \\
\overline{\chi}&3^k(q-1)         & && \varepsilon3^k &\varepsilon3^k\omega& \varepsilon3^k\overline{\omega}& && \\
\psi  &q(q-1)         & & &0&0&0&&&      \\
\end{array}\normalsize\]
for some fixed $\varepsilon \in \{\pm1 \}$, and where $q=3^{2k+1}$, $\omega=e^{2i\pi/3}$, and $\xi$ is a fixed primitive $(q-1)^{th}$ root of unity.

\vspace{5mm}
\noindent
In order to fill in the remaining entries, we apply the orthogonality relations of a character table.
First we consider the characters $\psi,\lambda$ and $\chi$ which arise as induced characters.
Therefore $\psi(g)=\lambda(g)=\chi(g)=\bar{\chi}(g)=0$ for all $g\in N_G(P)\setminus P$.

\vspace{2mm}
\noindent
\underline{\bf The characters $\psi$, $\lambda$ and $\chi$}

\vspace{1mm}
As $\psi|_P=\sum_{i=2}^{q} \psi_i$ and $\psi_i(g)=0$ for all $g\in P\setminus Z(P)$ \cite[p.70]{GramainPhD}, it follows that $\psi(g)=0$ for all $g\in P\setminus Z(P)$.
Moreover
\[
\begin{array}{rcl}
\langle \psi,\alpha_0\rangle& =\frac{1}{|N_G(P)|}( q(q-1) + \psi(X)(q-1)) =0
\end{array}
\]
implies $\psi(X)=-q$.

Now that we have $\psi$, both $\chi(X)$ and $\lambda(X)$ can be computed:
\[
\langle 0=\theta,\psi\rangle= \left\{
                   \begin{array}{ll}
                     \frac{1}{|N_G(P)|}( q(q-1)(q-1) -q \lambda(X)(q-1)) & \theta=\lambda;\\
                     \frac{1}{|N_G(P)|}( q(q-1)3^k(q-1) -q(q-1)\chi(X)) & \theta=\chi.\\
                    \end{array}
                   \right.
\]
Thus $\lambda(X)=q-1$ and $\chi(X)=\overline{\chi}(X)=3^k(q-1)$.
Therefore it remains to compute $\lambda(T)$ and $\chi(T)$.
Since $\lambda$ is the unique character taking the value $-1$ on $Y$,  $\lambda=\overline{\lambda}$.
In other words $\lambda$ is real valued and so $\lambda(T^{-1})=\overline{\lambda(T)}=\lambda(T)$.
Hence,

\[
\langle \lambda,\alpha_0\rangle= \frac{1}{|N_G(P)|}\left( q-1 + (q-1)^2+2\lambda(T)\frac{q(q-1)}{2}- 3\frac{q^2(q-1)}{3}\right)=0
\]
and it follows that $\lambda(T)=\lambda(T^{-1})=q-1$.

To compute $\chi(T)$, we need to make use of two relations:
\[
\langle \chi,\alpha_0\rangle =\frac{1}{|N_G(P)|}\left( 3^k(q-1)(1+q-1)+\frac{q(q-1)}{2}(\chi(T)+\overline{\chi(T)})+ \frac{q^2(q-1)}{3} \cdot\varepsilon\cdot 3^k(1+\omega+\overline{\omega})\right)=0
\]
and
\[
\langle \chi,\chi\rangle=\frac{1}{|N_G(P)|}\left( (3^k(q-1))^2(1+q-1)+2\chi(T)\overline{\chi(T)}\frac{q(q-1)}{2}+\frac{q^2(q-1)}{3}\cdot(\varepsilon\cdot 3^k)^2(1+\omega\overline{\omega}+\overline{\omega}\omega)\right)=1.
\]

Hence $\chi(T)+\overline{\chi(T)}=-2\cdot 3^k$ and $\chi(T)\overline{\chi(T)}=q(q-2\cdot 3^{2k})+3^{2k}=3^{2k}(q+1)$, where the finally equality comes from substituting in $q=3^{2k+1}$.
Thus $\chi(T)$ and $\overline{\chi(T)}=\chi(T^{-1})$ are the zeros of the polynomial $x^2+(\chi(T)+\overline{\chi(T)})x+\chi(T)\overline{\chi(T)}$.
In particular, $\chi(T)=-3^k + 3^{2k} \sqrt{-3}$ and $\chi(T^{-1})=  -3^k- 3^{2k} \sqrt{-3}$.

\vspace{5mm}\noindent
So far we have calculated the following additional entries:
%\begin{landscape}
\[\footnotesize
\begin{array}{c||ccccccccccccccc}
|\C(g)|&q-1&\frac{q(q-1)}{2}&\frac{q(q-1)}{2}  &q^3&q^2& \frac{q^2(q-1)}{2}& \frac{q^2(q-1)}{2}\\[7pt]
|C_N(g)|&q^3&2q^2&2q^2&q-1&q(q-1)& 2q&2q\\
\\\hline
&X&T&T^{-1}&Ph(w)&J& JT& JT^{-1}\\\hline\hline
\lambda&q-1        &q-1 &q-1& 0&0&0  &0  \\
%\mu_1&\frac{3^k(q-1)}{2}       & &  & & -\varepsilon3^k& -\varepsilon3^k \overline{\omega}& -\varepsilon 3^k\omega & && & \\
%\mu_2&\frac{3^k(q-1)}{2}       &  &  & & -\varepsilon3^k& -\varepsilon3^k \overline{\omega}& -\varepsilon 3^k\omega & && &  \\
%\overline{\mu_1}&\frac{3^k(q-1)}{2}       &  &   & & -\varepsilon3^k& -\varepsilon 3^k\omega& -\varepsilon 3^k\overline{\omega}  & && &\\
%\overline{\mu_2}&\frac{3^k(q-1)}{2}       &  &   & & -\varepsilon3^k& -\varepsilon 3^k\omega& -\varepsilon 3^k\overline{\omega}  & &&&\\
\chi       & 3^k(q-1)   &-3^k+ 3^{2k} \sqrt{-3} &-3^k- 3^{2k} \sqrt{-3}&  0 &0&0   &0 \\
\overline{\chi}&  3^k(q-1) &-3^k- 3^{2k} \sqrt{-3} &-3^k+ 3^{2k} \sqrt{-3}& 0 &0& 0&0 \\
\psi       &-q   & 0&0 &0&0&0   &0   \\
\end{array}\]\normalsize
where $q=3^{2k+1}$.
%\end{landscape}

\vspace{2mm}
\noindent
\underline{\bf The characters $\mu_i$ and $\overline{\mu_i}$ for $i\in\{1,2\}$} 

As before, to determine $\mu_i(X)$, we take the inner product of $\mu_i$ and $\psi$:
\[
\langle \mu_i,\psi\rangle=\frac{1}{|N_G(P)|}\left( q(q-1)\frac{3^k(q-1)}{2} -q(q-1)\mu_i(X)\right) =0
\] 
and thus $\mu_i(X)=\frac{3^k(q-1)}{2}$.

Consider $\mu_i(Ph(w_j))$, for $w_j\ne\pm 1$.
Since $4$ does not divide $q-1$, $\alpha_k(Ph(w_j))$ takes at least 5 distinct values for $0\leq k\leq q-2$.
Therefore if $\mu_i (Ph(w_j))\ne 0$ then the set  $\{\alpha_k\mu_i \mid 0\leq k\leq q-2\}$ must contain 5 distinct irreducible characters of degree equal to that of $\mu_i$, which is a contradiction.
Hence $\mu_1(Ph(w_j))= \mu_2(Ph(w_j))=\overline{\mu_1}(Ph(w_j))=\overline{\mu_2}(Ph(w_j))=0$.

As $\C(J)=\C(J^{-1})$, it follows that $\mu_i(J)\in \mathbb{R}$.
Moreover $JT$ and $JT^{-1}$ are inverses and hence $\mu_i(JT^{-1})=\overline{\mu_i}(JT)$.
Furthermore $\mu_1+\mu_2=Ind_P^{N_G(P)}\theta$ which evaluates to zero on $J, JT$ and $JT^{-1}$, and so $\mu_2(g)=-\mu_1(g)$, for $g=J,JT$ or $JT^{-1}$.
We have the following part of the character table:
\[\begin{array}{c||rrr}
&  J&JT&JT^{-1}\\\hline
& & & \\[-5pt]
\alpha_i& (-1)^i&(-1)^i&(-1)^i\\
\lambda& 0&0&0\\
\mu_1& c& b & \overline{b}\\
\mu_2& -c& -b & -\overline{b}\\
\overline{\mu_1}& c& \overline{b}& b\\
\overline{\mu_2}& -c& -\overline{b}&- b\\
\chi,\overline{\chi},\psi&0&0&0\\
\end{array}
\]
 where $c\in\mathbb{R}$ and $b\in \mathbb{C}\setminus \mathbb{R}$.

By column orthogonality 
\[
\sum\limits_{\theta\in {\rm Irr}(N_G(P))} \theta(J)^2= q-1 + 4c^2= q(q-1),
\]
and thus $c=(q-1)/2$.

By using column orthogonality again 
\[
\sum\limits_{\theta\in {\rm Irr}(N_G(P))} \theta(JT)\overline{\theta(JT)}= q-1+4 (b\overline{b}) = 2q
\]
 and so $b\overline{b}= \frac{q+1}{4}$.
However using column orthogonality for $JT$ and $J$ yields
\[
\sum\limits_{\theta\in {\rm Irr}(N_G(P))} \theta(JT)\theta(J)= q-1+ b(q-1)+\overline{b}(q-1) = 0
\]
and so $b+\overline{b}= -1$.
Thus as with $\chi$ above, $b$ and $\overline{b}$ are the zeros of the polynomial $x^2+(b+\overline{b})x+m\overline{b}$ and so it follows that $b= \frac{-1-3^k\sqrt{-3}}{2}$, by relabeling $T$ and $T^{-1}$ if necessary.

It only remains to compute $\mu_i$ on $T$ and $T^{-1}$.
Similar to the calculations for $\chi$, 
\[
\langle \mu_1,\theta\rangle \Rightarrow \left\{
                   \begin{array}{ll}
                    \mu_1(T)+\overline{\mu_1(T)}= -3^k & \text{ when }\theta=\alpha_0\\
                     \mu_1(T)\overline{\mu_1(T)}= \frac{1}{4}(3^{4k+1}+3^{2k})& \text{ when } \theta=\mu_1\\
                    \end{array}
                   \right.
\]
and solving the quadratic polynomial as before yields $\mu_1(T)= \frac{3^k\pm 3^{2k} \sqrt{-3}}{2}$.

Taking column orthogonality for $T$ and $J$, we see that $\mu_1(T)+\overline\mu_1(T)=\mu_2(T)+\overline\mu_2(T)$.
Row orthogonality implies $\langle \mu_1,\mu_1\rangle=\langle \mu_2,\mu_2\rangle=1$.
However as $\mu_2=-\mu_1$ on $J,JT$ and $JT^{-1}$, it follows that $\mu_1(T)\overline\mu_1(T)=\mu_2(T)\overline\mu_2(T)$.
Hence $\mu_2(T)=\mu_1(T)$ or $\overline\mu_1(T)$.
Finally, by column orthogonality for $T$ and $JT$, we see that $(\mu_1(T)-\mu_2(T))\overline{b}=(\overline\mu_2(T)-\overline\mu_1(T))b$ and therefore it follows that if $\mu_2(T)=\overline\mu_1(T)$ then $\overline{b}=-b$, which is a contradiction.
Thus $\mu_1(T)=\mu_2(T)$ and $\mu_1(T)=\mu_2(T)=\frac{-3^k+3^{2k} \sqrt{-3}}{2}$.

\vspace{10mm}\noindent
Combining all the character values together, we have proven the following theorem.
\begin{theorem}\label{th:ReeNormFullCharTab} Let $G=$ $^2G_2(q)$ where $q=3^{2k+1}$,  and let $N_G(P)$ be the normaliser of a Sylow $3$-subgroup. Then the character table of $N_G(P)$ is given by Table \ref{tb:CharacterTabReeNorm}.
\end{theorem}

\begin{landscape}
\begin{table}[c]
\centering
\caption{Character table of $N_G(P)$}\label{tb:CharacterTabReeNorm}
\[
\begin{array}{c||ccccccccccccccc}
%|\C(g)|& 1&q-1&\frac{q(q-1)}{2}&\frac{q(q-1)}{2}  &\frac{q^2(q-1)}{3}&\frac{q^2(q-1)}{3}& \frac{ q^2(q-1)}{3} &q^3&q^2& \frac{q^2(q-1)}{2}& \frac{q^2(q-1)}{2}\\[7pt]
|C_N(g)|& q^3(q-1)&q^3&2q^2&2q^2&3q&3q&3q&q-1&q(q-1)& 2q&2q\\\hline
 & 1&X&T&T^{-1}&Y&YT& YT^{-1}&Ph(w_j)&J& JT& JT^{-1}\\\hline\hline
\alpha_0=1_N&1       & 1 & 1&1&1& 1&1&1&1&1  &1  \\
\alpha_i \text{ for } 1\leq i\leq q-2&1     & 1 &1&1& 1& 1&1&\xi^{ij} &(-1)^i& (-1)^i &(-1)^i \\
\lambda&q-1       &q-1 &q-1 &q-1& -1&-1&-1&0&0&0  &0  \\
\mu_1&\frac{3^k(q-1)}{2}       &\frac{3^k(q-1)}{2} & \frac{-3^k+ 3^{2k} \sqrt{-3}}{2} &\frac{-3^k- 3^{2k} \sqrt{-3}}{2} & -\varepsilon3^k& -\varepsilon3^k \overline{\omega}& -\varepsilon 3^k\omega &0 & \frac{q-1}{2}&\frac{-1-3^k\sqrt{-3}}{2} & \frac{-1+3^k\sqrt{-3}}{2}\\
\mu_2&\frac{3^k(q-1)}{2}       & \frac{3^k(q-1)}{2} & \frac{-3^k+ 3^{2k} \sqrt{-3}}{2} &\frac{-3^k- 3^{2k} \sqrt{-3}}{2} & -\varepsilon3^k& -\varepsilon3^k \overline{\omega}& -\varepsilon 3^k\omega& 0&-\frac{q-1}{2}& \frac{+1+3^k\sqrt{-3}}{2}& \frac{+1-3^k\sqrt{-3}}{2} \\
\overline{\mu_1}&\frac{3^k(q-1)}{2}&\frac{3^k(q-1)}{2}&\frac{-3^k- 3^{2k} \sqrt{-3}}{2}& \frac{-3^k+ 3^{2k} \sqrt{-3}}{2}& -\varepsilon3^k& -\varepsilon 3^k\omega& -\varepsilon 3^k\overline{\omega} & 0&\frac{q-1}{2}&\frac{-1+3^k\sqrt{-3}}{2} &\frac{-1-3^k\sqrt{-3}}{2} \\
\overline{\mu_2}&\frac{3^k(q-1)}{2} &\frac{3^k(q-1)}{2}&\frac{-3^k- 3^{2k} \sqrt{-3}}{2}&\frac{-3^k+ 3^{2k} \sqrt{-3}}{2}& -\varepsilon3^k& -\varepsilon 3^k\omega& -\varepsilon 3^k\overline{\omega} &0 &-\frac{q-1}{2}&\frac{+1-3^k\sqrt{-3}}{2} & \frac{+1+3^k\sqrt{-3}}{2}\\
\chi &3^k(q-1)      & 3^k(q-1)   &-3^k+ 3^{2k} \sqrt{-3} &-3^k- 3^{2k} \sqrt{-3}&  \varepsilon3^k&\varepsilon3^k \overline{\omega}& \varepsilon3^k\omega&0 &0&0   &0 \\
\overline{\chi}&3^k(q-1)      &  3^k(q-1) &-3^k- 3^{2k} \sqrt{-3} &-3^k+ 3^{2k} \sqrt{-3}& \varepsilon3^k &\varepsilon3^k\omega& \varepsilon3^k\overline{\omega}&0 &0& 0&0 \\
\psi  &q(q-1)       &-q   & 0&0 &0&0&0&0&0&0   &0   \\
\end{array}
\]
where $q=3^{2k+1}$, $\varepsilon$ a fixed number from $\{\pm1 \}$, $\omega=e^{2i\pi/3}$ and $\xi$ is a fixed primitive $(q-1)^{th}$ root of unity.

\end{table}
\end{landscape}

\section{The principal $3$-block of the Ree groups}

Throughout this section, $G$ will denote the small Ree group $^2G_2(q)$ of order $q^3(q^3+1)(q-1)$ \cite{Ree}, with $q=3^{2k+1}\geq 27$. 
The groups $G$ were first described by Ree \cite{Ree} and their structure was determined in detail by Ward \cite{Ward}, including most of the character table. 
Our notation for the conjugacy classes and characters follows the notation introduced in \cite{Ward}; in particular, $m$ denotes the number $3^k$.
For the convenience of the reader the character table of $^2G_2(q)$, as contained in \cite{Ward}, is given in Table~\ref{tb:CharacterTabReeGp}, which can be found at the end of this paper.

For $k$ an algebraically closed field of characteristic $3$, the group algebra $kG$ decomposes into two blocks: the principal $3$-block $B_0(kG)$ and one block of defect zero containing the Steinberg character $\xi_3$ of degree $q^3$ (see Table \ref{tb:ReeCharacters}). 

\begin{table}[h]
\centering
\caption{The irreducible characters of $G$ \cite{Ward}}\label{tb:ReeCharacters}
\small
\[\begin{array}{ccc|lcccc}
& & \text{number of}&&& \text{number of}\\
\theta & \theta(1_G) &  \text{ characters} &\theta & \theta(1_G) & \text{ characters} \\\hline
%&&&&&&&\\
\xi_1& 1& 1 & \xi_9&m(q^2-1)& 1\\
\xi_2&q^2-q+1& 1& \xi_{10}&m(q^2-1)& 1\\
\xi_3&q^3& 1& \eta_r&q^3+1& (q-3)/4\\
\xi_4&q(q^2-q+1)& 1 & \eta'_r&q^3+1& (q-3)/4 \\
\xi_5&(q-1)m(q+1+3m)/2& 1 & \eta_t&(q-1)(q^2-q+1)& (q-3)/24   \\
\xi_6&(q-1)m(q+1-3m)/2& 1 & \eta'_t&(q-1)(q^2-q+1)& (q-3)/8   \\
\xi_7&(q-1)m(q+1+3m)/2& 1 & \eta^{-}_i&(q^2-1)(q+1+3m)& (q-3m)/6  \\
\xi_8&(q-1)m(q+1-3m)/2& 1 & \eta^{+}_i&(q^2-1)(q+1-3m)& (q+3m)/6  \\
\end{array}
\]\normalsize \end{table}

\begin{table}[h]
\centering
\caption{Conjugacy classes, their centraliser orders and corresponding defects}\label{tb:ReeCentOrders}
\[
\begin{array}{lccrcrcccc}
\text{conjugacy class} &\text{ order$(g)$ } & |C_G(g)|& d_g\\\hline
1_G&1&                  |G| & 3(2m+1)\\
R^a,\; 1\leq a\leq  (q-3)/4   & (q-1)/2 & q-1&0\\
S^a, \;1\leq a\leq (q-3)/24    &(q+1)/4    &q+1& 0\\
V_i, \;1\leq i\leq   (q-3m)/6    &q-\sqrt{3q}+1&    q+1-3m&0\\
W_i, \;1\leq i\leq  (q-3m)/6     &q+\sqrt{3q}+1&  q+1+3m&0\\
X&3&  q^3&3(2m+1)\\
Y&9&    3q&(2m+1)+1\\
T&3&    2q^2&2(2m+1)\\
T^{-1} &3&  2q^2&2(2m+1)\\
YT&9&   3q& (2m+1)+1\\
YT^{-1}&9&         3q& (2m+1)+1\\
JT&6&          2q &(2m+1)\\
JT^{-1}&6&       2q&(2m+1)\\
JR^a, \; 1\leq a\leq  (q-3)/4       &2|R^a|  &q-1&0\\
JS^a,  \; 1\leq a\leq  (q-3)/8       &2|S^a|  &q+1&0\\
J&2&  (q+1)q(q-1)&(2m+1)\\
\end{array}\] \end{table}

Table \ref{tb:ReeCentOrders} lists the conjugacy classes of $G$.
Note that for the six families of conjugacy classes, given by $R^a$, $S^a$, $V_i$, $W_i$, $JR^a$, $JS^a$, the elements don't have the order stated, but rather their order divides the given value.
Furthermore, the size of the centralisers for these elements are taken from \cite[Section 3]{Jones}.
Let
\[
\mathcal{S}=\{R^a, S^a, V_i,W_i, JR^a, JS^a\};
\]
$|\mathcal{S}|=q-2$.
With slight abuse of notation, we also say that a conjugacy class $\C(g)$ is in $\mathcal{S}$ or a conjugacy class sum $\Chat(g) \in \mathcal{S}$, if we want to refer to one of these families of conjugacy classes.

\vspace{2mm}
As mentioned above, Ward's character table \cite{Ward} is not quite complete. 
Some entries are missing; however since only their sum is of interest to us, orthogonality of the columns can be applied to find the required information. 
The following table provides a list of column orthogonality relations we consider and the implied relation upon character values.

\begin{table}[h]
\centering
\caption{Useful orthogonality relations}\label{tb:UsefulOrthRel}
\[
\begin{array}{cl|l}
& \text{$\C_1$ and $\C_2$  } & \text{ Implication from column orthogonality }\\[7pt] \hline
& R^a,JT & \sum_r \eta_r(R^a)= \sum_r \eta_r'(R^a)\\[7pt]
& R^a, 1_G& \sum_r \eta_r(R^a)=-1\\[7pt] \hline
& S^a,Y& 4- \sum_t \eta_t(S^a)-\sum_t \eta_t'(S^a)=0\\[7pt]
\text{combined with} & S^a,J& \sum_t \eta_t(S^a)=1,\;\; \sum_t \eta_t'(S^a)=3\\[7pt] \hline
& V_i,Y & \sum_i \eta_i^{-}(V_i)=1 \\[7pt] \hline
& W_i,Y & \sum_i \eta_i^{+}(W_i)=1 \\[7pt] \hline
& JR^a,1 &  \sum_r \eta_r(JR^a)= - \sum_r \eta_r'(JR^a)\\[7pt]
\text{combined with} & JR^a,JT& \sum_r \eta_r(JR^a)=-1,\;\;\sum_r \eta_r'(JR^a)=1\\[7pt]\hline
& JS^a,Y& - \sum_t \eta_t(JS^a)-\sum_t \eta_t'(JS^a)=0\\[7pt]
\text{combined with} & JS^a,J& \sum_t \eta_t(JS^a)=1,\;\; \sum_t \eta_t'(S^a)=-1\\[7pt] \hline
\end{array}
\]
\end{table}

We now write down the two block idempotents, where the equivalence is taken modulo  $J(\mathcal{O})G$.
\[
\begin{array}{ccl}
\hat{e}_{\xi_3}&=&\frac{\xi_3(1)}{|G|}\sum\limits_{g\in G} \xi_3(g^{-1}) g\\
&=&\frac{q^3}{q^3(q^3+1)(q-1)}\left(q^3+ \sum\limits_{a=1}^{(q-3)/4} \Chat(R^a)- \sum\limits_{a=1}^{(q-3)/24} \Chat(S^a) - \sum\limits_{i=1}^{(q-3m)/6} \Chat(V_i) \right.\\
&&\;\;\;\; \;\;\;\;\left.  - \sum\limits_{i=1}^{(q+3m)/6} \Chat(W_i)  + \sum\limits_{a=1}^{(q-3)/4} \Chat(JR^a)-\sum\limits_{a=1}^{(q-3)/8} \Chat(JS^a) +q\cdot\Chat(J) \right)\\
&\equiv& -\sum\limits_{a=1}^{(q-3)/4} \Chat(R^a)+ \sum\limits_{a=1}^{(q-3)/24} \Chat(S^a) + \sum\limits_{i=1}^{(q-3m)/6} \Chat(V_i)+ \sum\limits_{i=1}^{(q+3m)/6} \Chat(W_i)\\
&&\;\;\;\;\hspace{70mm} - \sum\limits_{a=1}^{(q-3)/4} \Chat(JR^a)+\sum\limits_{a=1}^{(q-3)/8} \Chat(JS^a)      \\
\end{array}
\]
and as $1=\hat{e}_0+\hat{e}_{\xi_3}$, it follows that $\hat{e}_0= 1-\hat{e}_{\xi_3} \;\;\in \mathcal{O}G$.

\subsection{Main Theorem for Ree groups}\label{Sec:MainThRee}

The aim is to show that the Loewy length of $Z(kGe_0)$ is $3$; in particular, over a series of lemmas comprising the rest of this section, the following theorem is proven.

\begin{theorem}\label{th:MainReeLL3}
Let $G=$ $^2G_2(q)$ where $q=3^{2k+1}\geq 27$, and $k$ an algebraically closed field of characteristic $3$. Then $LL(Z(kGe_0))=3$.
\end{theorem}

\vspace{2mm}
By Table \ref{tb:ReeCentOrders}, all non-trivial conjugacy classes have class size divisible by $3$ except $\C(X)$ which has size $|\C(X)|= (q^3+1)(q-1)$.
Hence a spanning set for $J(Z(kGe_0))$ is given by
\[ 
\mathfrak{D}_G=\{ \Chat(x)e_0\;|\; x\in \mathscr{P},  x\neq 1_G,x\not\in \C(X)\} \cup \{(\Chat(X)+1)e_0\}.
\]
Firstly, the multiplication of two  conjugacy class sums in $\mathfrak{D}_G$ is considered, ignoring the block idempotent $e_0$. 
For a clear overview, the outcomes of these multiplications are summarised in Table \ref{tb:SummaryClassMultRee}.

\subsubsection{Computing the product of any two class sums}

In this section we compute all the products of two class sums.
For most of the algebra structure constants we manipulate the formula of Burnside in such a way so that a minimal amount of computation has to be done; however some of the constants will require a complete calculation of Burnside's formula.
Let $cc(G)$ denote the set of conjugacy classes of $G$.

Before we compute the algebra structure constants, we first need the following lemma which helps us consider the terms arising in $a(x,y,z)$ from the characters $\eta_r,\eta_r',\eta_t,\eta_t'\eta_l^-$ and $\eta_l^+$ as in Table~\ref{tb:CharacterTabReeGp} taken from Ward \cite{Ward}.

\begin{lemma}\label{etaChar}
Let $\C(y)\in \mathcal{S}$ and $x,z\in$ $^2G_2(q)$.
Then
\[
\sum\limits_r \eta_r(x)\eta_r(y)\eta_r(z^{-1})\;\;, \;\sum\limits_r \eta_r'(x)\eta_r'(y)\eta_r'(z^{-1})\;\;, \; \sum\limits_t \eta_t(x)\eta_t(y)\eta_t(z^{-1}).
\]
\[
\sum\limits_t \eta_t'(x)\eta_t'(y)\eta_t'(z^{-1})\;\;, \;\sum\limits_l \eta_l^-(x)\eta_l^-(y)\eta_l^-(z^{-1})\;\;\text{ and } \; \sum\limits_l \eta_l^+(x)\eta_l^+(y)\eta_l^+(z^{-1}).
\]
all lie in $\mathbb{Z}$.
\begin{proof}
We shall only consider the case $\C(y)=\C(R^a)$ as the other cases follow by similar arguments.
Furthermore, we shall only calculate the following sum for $\eta_r$ as the situation for $\eta_r'$ is similar:
\begin{equation}\label{eq:eta}
\sum\limits_r \eta_r(x)\eta_r(R^a)\eta_r(z^{-1}).
\end{equation}

Note that from the character table, Table~\ref{tb:CharacterTabReeGp}, the characters $\eta_r$ takes values inside $\mathbb{Z}$ on elements not lying in $\C(R^a)$ or $\C(JR^a)$.
Furthermore from \cite[Section I]{Ward} it follows that $\eta_r(JR^a)=\eta_r(R^a)$ or $-\eta_r(R^a)$ and $\eta_r(R^a)=\epsilon (r^a+r^{-a})$ where $\epsilon=\pm 1$. 
Using this we obtain the following possibilities for Equation~\ref{eq:eta}.

If $x$ and $z\not\in \cup_a\left( \C(R^a)\cup\C(JR^a)\right)$ then
\[
\sum\limits_r \eta_r(x)\eta_r(R^a)\eta_r(z^{-1})= n\sum\limits_r \eta_r(R^a),
\]
for some $n\in \mathbb{Z}$.

If only one of $x$ or $z$ lies in some $\C(R^{a_1})\cup\C(JR^{a_1})$ then
\[
\sum\limits_r \eta_r(x)\eta_r(R^a)\eta_r(z^{-1})= n\sum\limits_r \big{(} \eta(R^{a+a_1})+\eta_r(R^{a-a_1})\big{)},
\] 
for some $n\in \mathbb{Z}$.

While if $x\in \C(R^{a_1})\cup\C(JR^{a_1})$ and $z\in \C(R^{a_2})\cup \C(JR^{a_2})$ then
\[
\sum\limits_r \eta_r(x)\eta_r(R^a)\eta_r(z^{-1})= n\sum\limits_r \big{(} \eta(R^{a+a_1+a_2})+\eta_r(R^{a+a_1-a_2})+\eta(R^{a-a_1+a_2})+\eta(R^{a-a_1-a_2})\big{)},
\] 
for some $n\in \mathbb{Z}$.

By Table~\ref{tb:UsefulOrthRel} it now follows that Equation~\ref{eq:eta} evaluates to an element in $\mathbb{Z}$.
\end{proof}
\end{lemma}

\begin{lemma}
Let $\C(x)\in cc(G)\setminus\{X,1_G\}$ and $\C(y)\in \mathcal{S}$. Then 
\[\Chat(x)\cdot \Chat(y)=
 \left\{
            \begin{array}{ll}
              e_{\xi_3} , & \hbox{ if $\C(x) \in \mathcal{S} \text{ or } \{\C(J)\}$;} \\
              0 , & \hbox{ otherwise.}
            \end{array}
          \right.
\]
\end{lemma}

\begin{proof}
First observe that $|C_G(y)|_3=1$ and since $q>3$, $|C_G(x)|_3\leq q^2$.
Thus
\[
\left| \frac{q^3(q-1)(q^3+1)}{|C_G(x)||C_G(y)|}\right| _3\geq q,
\]

For the characters $\xi_i$ with $i\ne 3$ or $4$, it can be seen that $|\xi_i(1)|_3<q$ and no value in the corresponding row contains a three in the denominator.
Thus any term in $a(x,y,z)$ arising from $\xi_i$ with $i\ne 3$ or $4$ is equivalent to zero modulo $J(\mathcal{O})G$.

Next we consider the terms in $a(x,y,z)$ which arise from the characters $\eta_r,\eta_r',\eta_t,\eta_t',\eta_l^-$ and $\eta_l^+$.
By combining Lemma~\ref{etaChar} with the additional fact that $|\theta(1)|_3=1$ for each such character $\theta$, the terms in $a(x,y,z)$ arising from these characters are equivalent to zero modulo $J(\mathcal{O})G$.
In particular, to compute $a(x,y,z)$ modulo $J(\mathcal{O})G$ only $\theta\in\{\xi_3,\xi_4\}$ remain to be considered.

If $\C(x)\not\in \{\mathcal{S}, 1_G,\C(J),\C(X)\}$, then $\xi_3(x)=\xi_4(x)=0$. 
Hence for $\C(x)\not\in \{\mathcal{S},1_G,\C(J),\C(X)\}$ it follows that $a(x,y,z)\equiv 0$ modulo $J(\mathcal{O})G$.

Thus is remains to study $\C(x)\in \{\mathcal{S},\C(J)\}$.
In this case $|C_G(x)|_3\leq q$.
Hence as 
\[
|\xi_4(1)|_3=q<q^2\leq \left| \frac{q^3(q-1)(q^3+1)}{|C_G(x)||C_G(y)|}\right| _3,
\]
the term for $\xi_4$ in $a(x,y,z)$ is congruent to zero modulo $J(\mathcal{O})G$.
Thus 
\[
a(x,y,z)\equiv \frac{q^3(q-1)(q^3+1)}{|C_G(x)||C_G(y)|} \left( \frac{\xi_3(x)\xi_3(y)\xi_3(z^{-1})}{q^3}\right).
\]

Note the following information about character values of $\mathcal{S}$.
\begin{table}[h]
\centering
\caption{}\label{tb:ValuesOnS}
\[
\begin{array}{r|c|r|c}
 &&& \text{coefficient of} \\
\text{element} & |C_G(g)|\; {\rm mod}\; 3 & \xi_3(g)&\Chat(g) \text{ in } e_{\xi_3} \\\hline
R^a&-1 &+1&-1\\
S^a&+1&-1&+1\\
V_i&+1&-1&+1\\
W_i&+1&-1&+1\\
JR^a&-1&+1&-1\\
JS^a&+1&-1&+1\\
\end{array}
\]
\end{table}
Also note that $\xi_3(J)=q$.

Hence
\[
a(x,y,z)\equiv \left\{ \begin{array}{ll}
                                 q^3\frac{(-1)}{(-1)}^a\cdot \frac{(-1)^a\xi_3(z^{-1})}{q^3} & \text{ if $\C(x)\in \mathcal{S}$};\\
                                q^2\frac{-1}{(-1)(\pm 1)}\cdot \frac{\mp q \xi_3(z^{-1})}{q^3} & \text{ if $\C(x)=\C(J)$}\\
                        \end{array}
                        \right.
\;\; \equiv -\xi_3(z^{-1}),
\]
for $a\in\{1,2\}$.

Finally, by considering the values $\xi_3$ takes, we see that
\[
a(x,y,z)\equiv -\xi_3(z^{-1})\equiv 
                                              \left\{ \begin{array}{ll}
                                              0 & \text{ if $z\not\in \mathcal{S}$};\\
                                             \text{coefficient of $\Chat(z^{-1})$ in $e_{\xi_3}$} & \text{ if $z\in \mathcal{S}$,}\\
                                              \end{array}
                                              \right.
\] 
where the values for $z\in \mathcal{S}$ follow from the information contained in Table~\ref{tb:ValuesOnS}.
\end{proof}

As we have computed $a(x,y,z)$ for $y\in \mathcal{S}$ for all $x$ except $a(X,y,z)$ Proposition~\ref{th:sumStrucConst} can now be used to find this final coefficient.

\begin{lemma}
Let $\C(y) \in \mathcal{S}$. 
Then $\Chat(X)\cdot \Chat(y)= e_{\xi_3}-\Chat(y)$. In particular
\[ a(X,y,z)=\left\{
              \begin{array}{rl}
                0 & \hbox{ if $y=z\in V_i,W_i,S^a,JS^a$;} \\
                -2& \hbox{ if  $y=z\in R^a, JR^a$; }\\
                \pm 1, & \hbox{ if $y\neq z$ and  $z\in \mathcal{S}$;} \\
                0, & \hbox{ if $z\not\in   \mathcal{S}$.}
              \end{array}
            \right.
\]
Hence $(1+\Chat(X))\cdot \Chat(y)= e_{\xi_3}$.
\end{lemma}
\begin{proof}
Recall  that $\mathcal{S}$ consists of $q-2$ conjugacy classes. For $y$ in $\mathcal{S} $ we have so far calculated all structure constants $a(x,y,z)$ except $a(X,y, z)$. Hence we can use Proposition \ref{th:sumStrucConst} to find these remaining ones. All equivalences are taken modulo $J(\mathcal{O})G$.

We have $|\C(y)|\equiv 0$. Hence by  Proposition \ref{th:sumStrucConst}, $|\C(y)|=\sum_{x} a(x,y,z)\equiv 0$. Let $\alpha= a(X,y,z)$. Then
\[\sum_{x} a(x,y,z)= a(1_G,y,z)+a(J,y,z)+(q-2)\sum\limits_{x\in\mathcal{S}}a(x, y, z)+\alpha\]

\[= \left\{
      \begin{array}{ll}
        1+1+(q-2)(+1)+ \alpha\equiv \alpha, & \hbox{ if $y=z\in V_i,W_i,S^a,JS^a$;} \\
        1-1+(q-2)(-1) +\alpha \equiv 2+\alpha, & \hbox{ if $y=z\in R^a, JR^a$;} \\
        0+1+(q-2)(+1)+\alpha\equiv -1+\alpha & \hbox{ if $y\neq z$ and $z\in V_i,W_i,S^a,JS^a$;} \\
        0+(-1)+(q-2)(-1)+\alpha \equiv \alpha+1 & \hbox{ if $y\neq z$ and $z\in R^a, JR^a$;} \\
        \alpha, & \hbox{ if $z\not\in\mathcal{S} $.}
      \end{array}
    \right.
\]\end{proof}

We have now dealt with the case that either $x$ or $y$ is taken from $\mathcal{S}$.

\begin{lemma}
Let $x,y\in \C(Y),\C(YT),\C(YT^{-1}),\C(JT),\C(JT^{-1})$ or $\C(J)$.
Furthermore assume that not both $x$ and $y$ are in $\C(J)$. 
Then $\Chat(x)\Chat(y)=0$.
\end{lemma}
\begin{proof}
We may assume that $\C(y)\ne \C(J)$.
By Table~\ref{tb:ReeCentOrders}, it follows that $|C_G(x)|_3$ and $|C_G(y)|_3\leq 3q$.
Hence
\[
\left| \frac{q^3(q-1)(q^3+1)}{|C_G(x)||C_G(y)|}\right| _3\geq \frac{q}{9}=3^{2k-1}.
\]

We assume first that $k\geq 2$.
In this case $3^{2k-1}>3^k$, and thus $|\theta(1)|_3<\frac{q}{9}$ for $\theta\in {\rm Irr}(G)$, unless $\theta=\xi_3$ or $\xi_4$.
In particular, the corresponding terms in $a(x,y,z)$ for $\theta\ne \xi_3,\xi_4$ are equal to zero modulo $J(\mathcal{O})G$.
Furthermore as $\C(y)\in \{\C(Y), \C(YT), \C(YT^{-1}), \C(JT), \C(JT^{-1})\} $, it can be seen that $\xi_3(y)=\xi_4(y)=0$.
Therefore we conclude that $a(x,y,z)\equiv 0$ modulo $J(\mathcal{O})G$.

Now consider the case that $k=1$.
If at most one of $x$ and $y$ lie in $\C(Y),\C(YT),\C(YT^{-1})$, then 
\[
\left| \frac{q^3(q-1)(q^3+1)}{|C_G(x)||C_G(y)|}\right| _3\geq \frac{q}{3}=3^{2k}.
\]
Hence the same argument for $k\geq 2$ holds.
Thus it remains to assume both $x$ and $y$ lie in $\C(Y),\C(YT)$ and $\C(YT^{-1})$.
As $k=1$, we can explicitly produce the character table of $^2G_2(q)$ in GAP \cite{GAP} and compute for each such $x$ and $y$ that $a(x,y,z)\equiv 0$ modulo $J(\mathcal{O})G$.
\end{proof}

For $y\in \C(Y),\C(YT),\C(YT^{-1}),\C(JT)$ or $\C(JT^{-1})$, the only remaining constants to evaluate are $a(T,y,z),a(T^{-1},y,z)$ and $a(X,y,z)$.
Lemma \ref{lm:inverseclass} reduces the number of computations required.
We state the following values without detail, these were evaluated by computing the full summand in Burnside's formula and are all given modulo $J(\mathcal{O})G$.
Note that the computations make use of Table~\ref{tb:UsefulOrthRel}, in the cases where $z\in\mathcal{S}$.
The full values can be found in \cite{PhDSchwabrow}.

\[\begin{array}{rclcl}
\Chat(T)\cdot\Chat(YT^{-1})&=& 0& = &\Chat(T^{-1})\cdot\Chat(YT)  \\
 \Chat(T)\cdot\Chat(YT)&=& 0 &=& \Chat(T^{-1})\cdot\Chat(YT^{-1})\\
\Chat(T)\cdot\Chat(JT^{-1})&=& 0         &=&  \Chat(T^{-1})\cdot\Chat(JT)   \\
\Chat(T)\cdot\Chat(JT)&=&0&=& \Chat(T^{-1})\cdot\Chat(JT^{-1}) \\
\Chat(Y)\cdot \Chat(T)&=& 0&=& \Chat(Y)\cdot \Chat(T^{-1})\\
\Chat(Y)\cdot \Chat(X)&=& 2\cdot\Chat(Y)\\
\end{array}\]

For $y\in \C(YT),\C(YT^{-1}),\C(JT)$ or $\C(JT^{-1})$ it remains to compute $a(X,y,z)$.

\begin{lemma}
Let $x\in\C(X)$   and  $y\in\C(Y),\C(YT)$,  $\C(YT^{-1})$,  $\C(JT)$ or $\C(JT^{-1})$. Then $\Chat(x)\cdot \Chat(y)= 2\cdot \Chat(y)$.
\end{lemma}
\begin{proof}
Fix $y\in \C(YT)$; the remaining three cases are proved in the same way.

By Proposition \ref{th:sumStrucConst}, $\sum_{x} a(x,YT,z)= |\C(YT)|= |G|/(3q)\equiv 0 \; {\rm mod}\;3$. Hence
\[ \left(\sum_{x\neq X} a(x,YT,z)\right) + a(X,YT,z)\equiv 0 \;{\rm mod}\; J(\mathcal{O})G\;\; \text{ for all }z\in G. \]
Now $a(x\neq X,YT,z)\equiv 0 $ except $a(1_G,YT,z)$. 
However, $a(1_G,YT,z)=1$ if $z=YT$ and zero otherwise.
Hence
\[a(X,YT,z) = \left\{
                \begin{array}{ll}
                  0, & \hbox{if $z \not\in \C(YT)$;} \\
                  2, & \hbox{if $z \in \C(YT)$.}
                \end{array}
              \right.
\]
\end{proof}

We now consider the case that both $x$ and $y$ lie in $\C(J)$.

\begin{lemma}
Let $x$ and $y$ both lie in $\C(J)$.
Then $\Chat(x)\cdot \Chat(y)= \xi_3$.
\end{lemma}
\begin{proof}

In this case $|C_G(x)|_3=|C_G(y)|_3=q$.
Thus
\[
\left| \frac{q^3(q-1)(q^3+1)}{|C_G(x)||C_G(y)|}\right| _3=q.
\]
Hence for any $\chi\in {\rm Irr}(G)$ such that $|\chi(1)|_3<q$, the corresponding term in $a(x,y,z)$ reduces to zero modulo $J(\mathcal{O})G$.
Hence the only terms remaining in $a(x,y,z)$ modulo $J(\mathcal{O})G$ are from $\xi_3$ and $\xi_4$.
Using this we see that
\[
\begin{array}{rl}
a(J,J,z)= & q\cdot\frac{q^2-q+1}{(q-1)(q+1)}\left(  \frac{\xi_3(J)\xi_3(J)\xi_3(z^{-1})}{q^3} +  \frac{\xi_4(J)\xi_4(J)\xi_4(z^{-1})}{q(q^2-q+1)}\right)\\[10pt]
=& q\cdot\frac{q^2-q+1}{(q-1)(q+1)}\left(\frac{\xi_3(z^{-1})}{q}+q\cdot\frac{\xi_4(z^{-1})}{q^2-q+1}\right)\\[10pt]
\equiv & -\xi_3(z^{-1})\\[10pt]
\end{array}.
\]
By Table~\ref{tb:ValuesOnS}, we have that $\Chat(J)\Chat(J)=e_{\xi_3}$.
\end{proof}

 Thus for $y\in \C(J)$ it remains to consider $a(T,y,z), a(T^{-1},y,z)$ and $a(X,y,z)$.
In the following Lemma we deal with $T$ and $T^{-1}$.

\begin{lemma}\label{lm:JTcalculation}
Let $y\in \C(J)$ and $x\in \C(T)$ or $\C(T^{-1})$.  Then $\Chat(x)\cdot \Chat(y) =0$.
\end{lemma}

\begin{proof}
By applying Theorem \ref{th:defectConClass},
\[ a(x,y,z)\equiv 0 \;{\rm mod}\; J(\mathcal{O})G \text{ if } \;\; z\in \{ X,Y, T, T^{-1}, YT, YT^{-1}\};\]
while by  \cite[Lemma 4.1]{Jones},
\[ a(x,y,z)\equiv 0 \;{\rm mod}\; J(\mathcal{O})G \text{ if } \;\; z\in \{ R^a,S^a,V_i,W_i\}.\]
Note that the formula used in the calculations by Jones \cite{Jones}, and hence the structure constants calculated, differs up to a scalar; we can adjust appropriately by dividing by $|\C(z)|$.

\vspace{4mm}\noindent
This leaves $a(x,y,z)$ for $z\in JT, JT^{-1}, JR^a, JS^a$ which can be calculated directly from the character table. In particular
\[ \begin{array}{ccl}
a(T,J, JT)&=& 9\cdot m^4+3\cdot m^2;\\
a(T,J, JT^{-1})& =& 0;\\
a(T,J, JR^a)& =&\frac{9}{2} m^4-\frac{3}{2}m^2;\\
a(T,J, JS^a)& =&\frac{9}{2}  m^4+\frac{3}{2} m^2,\\
\end{array}\]
where $m=3^k$. 
The equality  $a(T^{-1},J,z)= a(T,J, z^{-1})$ proved in Lemma \ref{lm:inverseclass} then completes this proof.
\end{proof}

Now we only have to compute $a(X,J,z)$. 

\begin{lemma}
Let $y \in \C(J)$. Then $\Chat(X)\cdot \Chat(J)= e_{\xi_3}-\Chat(J)$.
\end{lemma}
\begin{proof}
By Proposition \ref{th:sumStrucConst}, $\sum_{x} a(x,J,z)= |\C(J)|= |G|/((q+1)q(q-1))\equiv 0 \; {\rm mod}\;3$. Hence
\[\sum_{x} a(x,J,z) = a(1_G,J,z)+a(J,J,z)+(q-2)a(x\in \mathcal{S},J,z)+a(X,J,z) \equiv 0\]
\[= \left\{
      \begin{array}{ll}
        1+0+(q-2)(0)+ a(X,J,z)\equiv 1+ a(X,J,z), & \hbox{ if $z\in J$;} \\
        0+1+(q-2)(+1) + a(X,J,z)\equiv -1+ a(X,J,z), & \hbox{ if $z \in V_i,W_i,S^a,JS^a$;} \\
        0 -1 +(q-2)(-1)+a(X,J,z)\equiv 1+ a(X,J,z) , & \hbox{ if $z\in R^a, JR^a$;} \\
        a(X,J,z), & \hbox{ if $z\not\in \mathcal{S}$, $\C(J)$.}
      \end{array}
    \right.
\]
\end{proof}

For the remaining structure constants, when $\C(x),\C(y)\in \{\C(T),\C(T^{-1}),\C(X)\}$, direct calculations from the character table are used.
As the full Burnside formula was computed, we only provide the final values modulo $ J(\mathcal{O})G$. 
As before, the computations make use of Table~\ref{tb:UsefulOrthRel}, in the cases where $z\in\mathcal{S}$.
The full values can be found in \cite{PhDSchwabrow}.

\[\begin{array}{lcl}
\Chat(X)\cdot\Chat(X)&=& 2+\Chat(X)+\sum \Chat(R^a)+\sum\Chat(S^a)+\sum\Chat(V_i)+\sum\Chat(W)\\
\Chat(T)\cdot \Chat(T)&=& 2\cdot\sum_a \Chat(R^a)+ \sum_a \Chat(JR^a) +2\cdot\sum_a \Chat(JS^a)\\
&=&\Chat(T^{-1})\cdot \Chat(T^{-1})\\
&=&\Chat(T)\cdot\Chat(T^{-1})\\
\Chat(X)\cdot\Chat(T)&=& 2\Chat(T)+2\sum \Chat(R^a) +\sum \Chat(JR^a)+2\sum \Chat(JS^a)\\
\Chat(X)\cdot\Chat(T^{-1})&=&2\Chat(T^{-1})+2\sum \Chat(R^a) +\sum \Chat(JR^a)+2\sum \Chat(JS^a)\\
\end{array}\]

\subsubsection{\underline{Summary of multiplications}}
\begin{table}[h]
\centering
\caption{Summary of multiplications of two conjugacy class sums in $kG$}\label{tb:SummaryClassMultRee}
\[\footnotesize
\begin{array}{l|cccccccccccccccc}
	&R^a & S^a&	\;V_i\;&\; W_i\; &\;X \;&\; Y\;&\;  T\;&   T^{-1} &YT&YT^{-1}& JT&JT^{-1} & JR^a &JS^a &J  \\  \hline
R^a& e_{\xi_3} &e_{\xi_3}&e_{\xi_3}&e_{\xi_3}&\gamma_1&-&-&-&-&-&-&-&e_{\xi_3}&e_{\xi_3}& e_{\xi_3}\\[7pt]
S^a	&    &e_{\xi_3}&e_{\xi_3}&e_{\xi_3}&\gamma_2&-&-&-&-&-&-&-&e_{\xi_3}&e_{\xi_3}&e_{\xi_3} \\[7pt]
V_i	&&&e_{\xi_3}  &e_{\xi_3}&\gamma_3&-&-&-&-&-&-&-&e_{\xi_3}&e_{\xi_3}&e_{\xi_3}\\[7pt]
W_i	&&&  &e_{\xi_3}&\gamma_4&-&-&-&-&-&-&-&e_{\xi_3}&e_{\xi_3}&e_{\xi_3}	\\[7pt]
X	&&&&&\alpha&\delta_1&\mu&\nu&\delta_2&\delta_3&\delta_4&\delta_5&\gamma_5&\gamma_6&\gamma_7	\\[7pt]
Y	&&&&&			&-&- &- &- &- &- &-&-&- &-\\[7pt]
T	 &	&  & 	 &	&	&		&\beta&\beta&-&-&- &-&-&- &-\\[7pt]
T^{-1}			&&&&&&&&\beta&-&-&-&- &- &-& -\\[7pt]
YT	 &	&  & 	 &	&	&		&	&    & -&-&-&-&-&-&-\\[7pt]
YT^{-1}		 &	&  & 	 &	&	&	&	&    & &-&-&-&-&-&-\\[7pt]
JT   &	&  & 	 &	&	&		&	&    &   &&-&-&-&-&-\\[7pt]
JT^{-1}	 &	&  & 	 &	&	&		&	&    &  & &   &-&-&-&-\\[7pt]
JR^a   &	&  & 	 &	&	&		&	&    &  & &   &  & e_{\xi_3}&e_{\xi_3}&e_{\xi_3} \\[7pt]		
JS^a   &	&  & 	 &	&	&		&	&    &  & &   &  &&e_{\xi_3}& e_{\xi_3}\\[7pt]	
J   &	&  & 	 &	&	&		&	&    &  & &   &  & &&e_{\xi_3}\\[7pt]				
\end{array}\normalsize
\]\end{table}
In Table~\ref{tb:SummaryClassMultRee} we use `` $-$ '' to denote a zero in $kG$ and
\[ \begin{array}{ccl}
\alpha&=& 2+\Chat(X)+\sum \Chat(R^a)+\sum \Chat(S^a) +\sum \Chat(V_i)+\sum \Chat(W_i)\\
\beta&= & 2\sum \Chat(R^a) +\sum \Chat(JR^a)+2\sum \Chat(JS^a) \\
\end{array}
\]
\[ \begin{array}{cclccl}
\gamma_1&=&e_{\xi_3}-\Chat(R^a)&\delta_1& = & 2\cdot\Chat(Y) \\
\gamma_2&=&e_{\xi_3}-\Chat(S^a)&\delta_2& = & 2\cdot\Chat(YT) \\
\gamma_3&=&e_{\xi_3}-\Chat(V_i)&\delta_3& = & 2\cdot\Chat(YT^{-1}) \\
\gamma_4&=&e_{\xi_3}-\Chat(W_i)&\delta_4& = & 2\cdot\Chat(JT) \\
\gamma_5&=&e_{\xi_3}-\Chat(JR^a)& \delta_5& = & 2\cdot\Chat(JT^{-1}) \\
\gamma_6&=&e_{\xi_3}-\Chat(JS^a)& \mu &=& 2\Chat(T)+2\sum \Chat(R^a) +\sum \Chat(JR^a)+2\sum \Chat(JS^a)\\
\gamma_7&=&e_{\xi_3}-\Chat(J)& \nu &=& 2\Chat(T^{-1})+2\sum \Chat(R^a) +\sum \Chat(JR^a)+2\sum \Chat(JS^a)\\
\end{array}
\]
Note that $\Chat(X)\Chat(y)=\gamma_i= e_{\xi_3}-\Chat(y)$ so that $(1+\Chat(X))\Chat(y)= e_{\xi_3}$. Moreover, $\C(X)\cdot \C(Y)=\delta_i=2\cdot \Chat(y)$ so that $(1+\Chat(X))\Chat(y)=0$.

Since $e_{\xi_3}\cdot e_0=0$, we can see from Table \ref{tb:SummaryClassMultRee} that most pairs of elements in  $\mathfrak{D}_G$ multiply to zero. However there exist elements $b, b'$ in $\mathfrak{D}_G$ such that $b\cdot b'\neq 0$:
\[\begin{array}{ccl}
(1+\Chat(X))^2e_0&=& \left(\sum \Chat(R^a)+\sum \Chat(S^a) +\sum \Chat(V_i)+\sum \Chat(W_i)\right)e_{0}\\[5pt]
&   =&2\sum\Chat(R^a)+\sum \Chat(JR^a)+2\sum\Chat(JS^a)\\
\end{array}\]

\[\begin{array}{ccl}
(1+\Chat(X))\Chat(T)e_0&=&  (1+\Chat(X))\Chat(T^{-1})e_0   =  (\Chat(T^{\pm 1}))^2e_0  = \Chat(T)\cdot \Chat(T^{-1})e_0    \\[5pt]
&=& 2\sum \Chat(R^a) +\sum \Chat(JR^a)+2\sum \Chat(JS^a) \\[5pt]
\end{array}\]

\subsection{The proof of Theorem~\ref{th:MainReeLL3}}
As all the products $\Chat(x)\Chat(y)$ have been computed, we can complete the proof of Theorem~\ref{th:MainReeLL3}.

{\bf Proof of Theorem~\ref{th:MainReeLL3}}.
By Table \ref{tb:SummaryClassMultRee} and the discussion below it, there exist elements $b,b'$ in $\mathfrak{D}_G$ such that $b\cdot b'\neq 0$. Hence $LL(Z(kGe_0))\geq 3$.

Note that the number of the conjugacy classes labeled by $\C(S^a)$, for some $a$, is the only one not congruent to zero modulo $3$; in fact, we  have $(q-3)/24\equiv 1$ modulo $3$ of those (see Table \ref{tb:ReeCentOrders}). This explains why $(1+\Chat(X))^2e_0 \neq (1+\Chat(X))^2$ while $(1+\Chat(X))\Chat(T^{\pm1})e_0=  (1+\Chat(X))\Chat(T^{\pm 1})$, and  $(\Chat(T^{\pm 1}))^2e_0= (\Chat(T^{\pm 1}))^2$. In particular,
\[ \left(\sum\limits_{a=1}^{(q-3)/24} \Chat(S^a)   \right)   \cdot  e_{\xi_3}\equiv \left(\sum\limits_{a=1}^{(q-3)/24} \Chat(S^a) \right) \cdot \left(\sum\limits_{a=1}^{(q-3)/24} \Chat(S^a) \right) \equiv e_{\xi_3} \neq 0.\]

From the multiplications already computed, it can be concluded that the Loewy length must be equal to $3$, since
none of the outcomes of the non-zero multiplications of two elements in $\mathfrak{D}_G$ involve the conjugacy classes $\C(X), \C(T), \C(T^{-1})$ or $\C(S^a)$.  It therefore follows that any triple of elements in $\mathfrak{D}_G$ will multiply to zero. Hence $LL(Z(kGe_0))=3$.
\qed

\section{The $3$-block of the Sylow normaliser}
We now state the main theorem on the normaliser of a Sylow $3$-subgroup. 
Recall that the group algebra $kN_G(P)$ is indecomposable, where $k$ is an algebraically closed field of characteristic $3$.
Throughout we shall assume that $q=3^{2k+1}$ with $k>0$.

\begin{theorem}\label{th:Snowman}
Let $N=N_G(P)$ where $G=$ $^2G_2(q)$, $q=3^{2k+1}\geq 27$, and $P\in {\rm Syl}_3(G)$. Then $LL(Z(kN))=2$.
\end{theorem}

\noindent
By Table \ref{tb:ReeNormConCl},  all non-trivial conjugacy classes of $N_G(P)$ have class size divisible by $3$ except $\C(X)$ which has size $|\C(X)|=q-1$.
Therefore a basis for $J(Z(kN_G(P)))$ is given by
\[ 
\mathfrak{B}_{N_G(P)}=\{ \Chat(x)\;|\; x\in \mathscr{P}, x\neq 1_{N_G(P)},x\not\in \C(X)\} \cup \{\Chat(X)+1\}.
\]
The proof of Theorem \ref{th:Snowman} will be spread over several lemmas. 
Let $cc(N_G(P))$ denote the set of conjugacy classes inside $N_G(P)$. All conjugacy class sums are multiplied as elements in $kN_G(P)$, and all equivalences are taken modulo $J(\mathcal{O})N_G(P)$.
We will leave the multiplications for the element $\Chat(X)+1$ till we have computed the other products.

Firstly we deal with the case where $x\in Ph(w_j)$ not corresponding to $J,JT$ or $JT^{-1}$.
\begin{lemma}\label{Phwj}
Let $\C(x)\in \{Ph(w_j)\mid w_j\ne\pm 1\}$ and $\C(y)\in cc(N_G(P))\setminus \{\Chat(1_N), \Chat(X)\}$. 
Then $\Chat(x)\cdot \Chat(y)=0$.
\end{lemma}
\begin{proof}
By Table \ref{tb:ReeNormConCl}, $|C_G(x)|= q-1$ and $|C_G(y)|_3 < q^3$. 
As the only characters which do not vanish on $Ph(w_j)$ are the linear characters $\alpha_i$ for $0\leq i\leq q-2$, it follows that for all $z\in N_G(P)$
\[
a(x,y,z)= 3^a\cdot \frac{q-1}{s} \left( \sum\limits_{i=0}^{q-2} \alpha_i(x)\alpha_i(y)\alpha_i(z^{-1})\right)= 3^a\cdot \frac{q-1}{s} \left( \sum\limits_{i=0}^{q-2} \alpha_i(x)\alpha_i(yz^{-1})\right)
\]
\[
= 3^a\cdot \frac{q-1}{s}(\delta_{x,(zy^{-1})}\cdot |C_{N_G(P)}(x)|),
\]
for $a\geq1$, ${\rm gcd}(3,s)=1$ and $\delta_{x,y}$ is defined to be equal to $1$ if $x=y$ and $0$ otherwise.
The second equality follows from the fact that degree one characters are representations of the group and the third from the column orthgonality.
Thus is follows that $a(x,y,z)\equiv 0$ modulo $J(\mathcal{O})N_G(P)$.
\end{proof}

As $a(x,y,z)=a(y,x,z)$ it shall be assumed from now on that neither $x$ nor $y$ is of the form $Ph(w_j)$, where $Ph(w_j)$ is not one of $J,JT,JT^{-1}$.
Next consider the conjugacy classes of $J,JT$ and $JT^{-1}$.

\begin{lemma}\label{JJTJT-1}
Let $x\in \C(J), \C(JT)$ or $\C(JT^{-1})$, and $\C(y)\in cc(N_G(P))\setminus \{\C(1_N), \C(X),\C(Ph(w_j)) \}$. Then $\Chat(x)\cdot \Chat(y)=0$.
\end{lemma}
\begin{proof}

First we observe that 
\[
\sum\limits_{i=0}^{q-2}\alpha_i(x)\alpha_i(y)\alpha_i(z^{-1})=\sum\limits_{i=0}^{q-2}(-1)^i\alpha_i(yz^{-1})=\left\{
                   \begin{array}{ll}
                     q-1 & yz^{-1}\in \C(J),\C(JT)\text{ or }\C(JT^{-1})\\
                     0 & otherwise\\
                    \end{array}
                   \right.;
\]
the final equality follows by taking row orthogonality in the character table of $N_G(P)/P$.

For $x$ as in the statement of the Lemma, we see that
\[
\begin{array}{rl}
a(x,y,z)=& \frac{q^3(q-1)}{|C_{N_G(P)}(x)||C_{N_G(P)}(y)|}
\left(
\sum\limits_{\theta\in\{\mu_i,\overline{\mu_i}\}}\frac{2\theta(x)\theta(y)\theta(z)}{3^k(q-1)} +\sum\limits_{i=0}^{q-2}(-1)^i\alpha_i(yz^{-1})
\right)\\
=& \frac{q^3(q-1)}{3^k(q-1)|C_{N_G(P)}(x)||C_{N_G(P)}(y)|}
\left(
\sum\limits_{\theta\in\{\mu_i,\overline{\mu_i}\}}2\theta(x)\theta(y)\theta(z) +3^k(q-1)\sum\limits_{i=0}^{q-2}(-1)^i\alpha_i(yz^{-1})
\right).\\
\end{array}
\]

As $\mu_i(g)=\frac{a+b\sqrt{-3}}{2}$ for $a,b\in \mathbb{Z}$ and the summands arising from the $\alpha_i$ add up to an element in $\mathbb{Z}$, it is enough to consider when the front coefficient, given by
\[
\frac{q^3(q-1)}{3^k(q-1)|C_{N_G(P)}(x)||C_{N_G(P)}(y)|},
\]
is divisible by $3$. 
However, as $|C_{N_G(P)}(x)|=qs$, where ${\rm gcd}(3,s)=1$, this coefficient reduces to 
\[
\frac{q^2}{3^ks|C_{N_G(P)}(y)|}.
\]
In particular, as $y\not\in \C(X)$, it follows that this coefficient is divisible by $3$ provided $y$ is not in $\C(T)$ or $\C(T^{-1})$.
Thus is remains to consider the cases $y\in \C(T)$ and $\C(T^{-1})$.

For $z\in \C(1_{N_G(P)}),C(X),\C(T),\C(T^{-1}), \C(Y),\C(YT)$ or $\C(YT^{-1})$, we have that $d_x<d_z$ and thus Theorem~\ref{th:defectConClass} implies $a(x,y,z)\equiv 0$.
If $z\in Ph(w_j)$ for $w_j\ne \pm1$, then 
\[
a(x,y,z)= 3^a\cdot \frac{q-1}{s} \left( \sum\limits_{i=0}^{q-2} \alpha_i(x)\alpha_i(yz^{-1})\right),
\]
where $a\geq 0$.
This sum is non-zero only if $yz^{-1}\in \C(J),\C(JT),\C(JT^{-1})$, which implies that $yz^{-1}$ lies in $Ph(-1)$.
As $y\in P$, by taking the image inside $N_G(P)/P=W$, it follows that $z$ must also lie in $Ph(-1)$, which is a contradiction.
Thus $a(x,y,z)=0$.

As $a(x,T,z)=a(x^{-1},T^{-1},z^{-1})$ and $\C(J)^{-1}=\C(J)$, $\C(JT)^{-1}=\C(JT^{-1})$, it is enough to consider $y\in \C(T)$ and $z\in \C(J),\C(JT)$ or $\C(JT^{-1})$.

In this case it follows that 
\[
\sum\limits_{i=0}^{q-2}\alpha_i(x)\alpha_i(T)\alpha_i(z)=\sum\limits_{i=0}^{q-2}\alpha_i(T)=q-1.
\]
Therefore, for $a=\frac{-3^k+3^{2k}\sqrt{-3}}{2}$ and $b=\frac{-1-3^k\sqrt{-3}}{2}$,
\[
\sum\limits_{\theta\in\{\mu_i,\overline{\mu_i}\}}2\theta(x)\theta(y)\theta(z) +3^k(q-1)^2=
\left\{
\begin{array}{ll}
4(\frac{q-1}{2})^2(a+\overline{a})+3^k(q-1)^2 & x,z\in\C(J)\\
4(\frac{q-1}{2})(a\overline{b}+\overline{a}b)+3^k(q-1)^2 & x\in \C(J),z\in \C(JT)\\
4(\frac{q-1}{2})(ab+\overline{a}\overline{b})+3^k(q-1)^2 & x\in \C(J),z\in \C(JT^{-1})\\
4(ba\overline{b}+\overline{b}\overline{a}b)+3^k(q-1)^2 & x\in \C(JT),z\in \C(JT)\\
4(bab+\overline{b}\overline{a}\overline{b})+3^k(q-1)^2 & x\in \C(JT),z\in \C(JT^{-1})\\
4(\overline{b}a\overline{b}+b\overline{a}b)+3^k(q-1)^2 & x\in \C(JT^{-1}),z\in \C(JT)\\
\end{array}
\right.
\]
Hence
\[
a(x,T,z)=
\left\{
\begin{array}{ll}
0 & x,z\in\C(J)\\
0 & x\in \C(J),z\in \C(JT)\\
q & x\in \C(J),z\in \C(JT^{-1})\\
\frac{q(q-3)}{4} & x\in \C(JT),z\in \C(JT)\\
\frac{q(q-3)}{4} & x\in \C(JT),z\in \C(JT^{-1})\\
\frac{q(q+1)}{4} & x\in \C(JT^{-1}),z\in \C(JT)\\
\end{array}
\right.
\]

By recalling that $a(x,y,z)= a(z^{-1},y,x^{-1})\frac{C_H(z)}{C_H(x)}$, we note that $a(JT,T,J)=\frac{2}{q-1}a(J,T,JT^{-1})$, $a(J,T,JT^{-1})=\frac{2}{q-1}a(J,T,JT)$ and $a(JT,T,JT)=a(JT^{-1},T,JT^{-1})$.
Thus it follows that $a(x,T,z)\equiv 0$ modulo $J(\mathcal{O})N_G(P)$.

This completes the proof.
\end{proof}

As before, we can now assume that neither $x$ nor $y$ lie in one of $Ph(w_j)$, $J$,$JT$ or $JT^{-1}$.

\begin{lemma}
Let $\C(x),\C(y)\in \{\C(Y), \C(YT), \C(YT^{-1})\}$. Then $\Chat(x)\cdot \Chat(y)=0$.
\begin{proof}
As both $\C(x)$ and $\C(y)$ lie in $P$, which is a normal subgroup, then $a(x,y,z)=0$ for $z\in \C(Ph(w_j))$, $\C(J),\C(JT)$ or $\C(JT^{-1})$.
In particular, we may now assume that $\alpha_i(x)=\alpha_i(y)=\alpha_i(z)=1$.

By Table \ref{tb:ReeNormConCl}, $|C_G(x)|= |C_G(y)|= 3q$.
Hence
\[
a(x,y,z)=\frac{q^3(q-1)}{3^2q^2 }\left(
\sum\limits_{\theta\in \mathfrak{A}} \frac{\theta(x)\theta(y)\theta(z^{-1})}{\theta(1)} + q-1+\frac{\lambda(z^{-1})}{q-1}
\right).
\]
where $\mathfrak{A}=\{\mu_1, \mu_2, \overline{\mu_1}, \overline{\mu_2}, \chi, \overline{\chi}\}$. 
Note that for $\theta\in \mathfrak{A}$, we have $|\theta(1)|_3=3^k$; however at the same time $3^k=|\theta(x)|_3=|\theta(y)|_3$. 
Hence  $a(x,y,z)\equiv 0$ for all $z\in N_G(P)$.
\end{proof}
\end{lemma}

\begin{lemma}
Let $\C(x)\in \{\C(T),\C(T^{-1})\}$ and $\C(y)\in \{\C(Y),\C( YT), \C(YT^{-1})\}$. Then $\Chat(x)\cdot \Chat(y)=0$.
\end{lemma}

\begin{proof}
As $x,y\in P$, which is a normal subgroup, then $a(x,y,z)=0$ for $z\in \C(Ph(w_j)), \C(J), \C(JT)$ or $\C(JT^{-1})$.
If $z\in \C(1_N), \C(X), \C(T)$ or $\C(T^{-1})$, then $d_y<d_z$ and so by Theorem \ref{th:defectConClass}, $a(x,y,z)\equiv 0$  mod $J(\mathcal{O})N_G(P)$.

Thus assume $z\in \C(Y),\C(YT$) or $\C(YT^{-1})$.
Then
\[
a(x,y,z)=\frac{(q-1)}{2\cdot3}\left((q-1)+1+\sum\limits_{\theta\in \mathfrak{A}} \frac{\theta(x)\theta(y)\theta(z^{-1})}{\theta(1)}\right),
\]
where $\mathfrak{A}=\{\mu_1, \mu_2, \overline{\mu_1}, \overline{\mu_2}, \chi, \overline{\chi}\}$. 
The multiplication $\theta(x)\theta(y)\theta(z^{-1})$ is of the form $3^{3k}\cdot s$, where ${\rm gcd}(3,s)=1$.  
Hence, as $k\geq 1$, we have that modulo $J(\mathcal{O})N_G(P)$
\[
a(x,y,z)= \frac{q-1}{2}\left( \frac{q}{3}+3^{2k-1}\cdot \frac{s_1}{s_2}\right)\equiv 0,
\]
as ${\rm gcd}(3,s_i)=1$.
\end{proof}

\begin{lemma}
Let $\C(x),\C(y)\in \{\C(T),\C(T^{-1})\}$. Then $\Chat(x)\cdot \Chat(y)=0$.
\end{lemma}
\begin{proof}
As in the previous lemma, if $z\in \C(Ph(w_j)), \C(J), \C(JT)$ or $\C(JT^{-1})$ then $a(x,y,z)=0$. 
If $z$ in $\C(1_N), \C(X)$  then $d_x<d_z$ and so by Theorem \ref{th:defectConClass}, $a(x,y,z)\equiv 0$.

If $\C(z)\in \{\C(T),\C(T^{-1}), \C(Y),\C(YT), \C(YT^{-1}) \}$ then by Proposition~\ref{prop:a(xyz)generic}, $a_{N_G(P)}(x,y,z)\equiv a_G(x,y,z) \equiv 0\;\;{\rm mod} \;J(\mathcal{O})N$.
\end{proof}

\begin{lemma} Let $\C(y)$ in $cc(N)\setminus \{\C(X),\C(1_N)\}$. Then $(\Chat(X)+1) \cdot \Chat(y)= 0$  and $ (\Chat(X)+1)^2=0$.
\end{lemma}
\begin{proof} So far we have already calculated all structure constants apart from $a(X,y,z)$; hence we can use Proposition \ref{th:sumStrucConst} to find the remaining ones.

By Proposition \ref{th:sumStrucConst} and the conjugacy class sizes given in Table \ref{tb:ReeNormConCl}, we have $\sum_x a(x,y,z)=|\C(y)|\equiv 0 \; {\rm modulo}\;3$. Hence modulo $J(\mathcal{O})N_G(P)$.
\[
a(1,y,z)+\sum_{x\not\in \{1_N,\C(X)\}}a(x,y,z)+ a(X,y,z)\equiv 0 \;\; \forall z\in N_G(P).
\]
Now $a(x\neq X,y,z)\equiv 0$ and $a(1,y,z)\equiv0$ except $a(1,y,y)=1$. 
Therefore
\[a(X,y,z) = \left\{
                \begin{array}{ll}
                  0, & \hbox{if $\C(y)\neq \C(z)$;} \\
                  2, & \hbox{ if $\C(y)=  \C(z)$.}
                \end{array}
              \right.
\]
Hence $\Chat(X) \cdot \Chat(y)=2\cdot \Chat(y)$ and so $(\Chat(X)+1) \cdot \Chat(y)\equiv 0$.

Moreover $1+\Chat(X) = \sum_{\gamma\in Z(P)}\gamma$, and therefore $(1+\Chat(X))^2=q(1+\Chat(X)) \equiv 0$.
\end{proof}

This concludes the proof of Theorem \ref{th:Snowman}.
Moreover, by combining Theorem~\ref{th:Snowman} and Theorem~\ref{th:MainReeLL3} we have proven the main result of this paper, Theorem~\ref{th:FullStatmentRee}.

\section*{Acknowledgments}

The work formed part of the second author's PhD research, which was supported by EPSRC grant $1240275$.
The research of the first author is supported by the LMS Postdoctoral Mobility Grant 15-16 08. 

\bibliographystyle{plain}
\bibliography{bibfile}

\begin{thebibliography}{10}

\bibitem{Blau}
H.~I. Blau and G.~O. Michler.
\newblock Modular representation theory of finite groups with {T}.{I}. {S}ylow
  {$p$}-subgroups.
\newblock {\em Trans. Amer. Math. Soc.}, 319(2):417--468, 1990.

\bibitem{Broue88}
M.~Brou{\'e}.
\newblock Blocs, isom\'etries parfaites, cat\'egories d\'eriv\'ees.
\newblock {\em C. R. Acad. Sci. Paris S\'er. I Math.}, 307(1):13--18, 1988.

\bibitem{Broue1}
M.~Brou{\'e}.
\newblock Isom\'etries parfaites, types de blocs, cat\'egories d\'eriv\'ees.
\newblock {\em Ast\'erisque}, (181-182):61--92, 1990.

\bibitem{Burnside}
W.~Burnside.
\newblock {\em Theory of groups of finite order}.
\newblock Dover Publications, Inc., New York, 1955.
\newblock 2d ed.

\bibitem{SGLT}
R.~Carter.
\newblock {\em Simple Groups of Lie Type}.
\newblock Wiley Classics Library, 1989.

\bibitem{Cliff}
G.~Cliff.
\newblock On centers of {$2$}-blocks of {S}uzuki groups.
\newblock {\em J. Algebra}, 226(1):74--90, 2000.

\bibitem{CurtisReiner}
C.~W. Curtis and I.~Reiner.
\newblock {\em Representation theory of finite groups and associative
  algebras}.
\newblock Pure and Applied Mathematics, Vol. XI. Interscience Publishers, a
  division of John Wiley \& Sons, New York-London, 1962.

\bibitem{Eaton1}
C.~W. Eaton.
\newblock Dade's inductive conjecture for the {R}ee groups of type {$G_2$} in
  the defining characteristic.
\newblock {\em J. Algebra}, 226(1):614--620, 2000.

\bibitem{CFSG}
D.~Gorenstein, R.~Lyons, and R.~Solomon.
\newblock {\em The classification of the finite simple groups. {N}umber 3.
  {P}art {I}. {C}hapter {A}}, volume~40 of {\em Mathematical Surveys and
  Monographs}.
\newblock American Mathematical Society, Providence, RI, 1998.
\newblock Almost simple $K$-groups.

\bibitem{GramainPhD}
J.B. Gramain.
\newblock {\em Generalized Block Theory}.
\newblock PhD Thesis, 2005.

\bibitem{GAP}
The~{GAP} group.
\newblock {\em GAP - groups, algorithms and programming, Version 4.7.9; 2015.
  (http://www.gap-system.org)}.

\bibitem{Jones}
G.~A. Jones.
\newblock Ree groups and {R}iemann surfaces.
\newblock {\em J. Algebra}, 165(1):41--62, 1994.

\bibitem{KarpilovskyJacRadical}
G.~Karpilovsky.
\newblock {\em The {J}acobson radical of group algebras}, volume 135 of {\em
  North-Holland Mathematics Studies}.
\newblock North-Holland Publishing Co., Amsterdam, 1987.
\newblock Notas de Matem{\'a}tica [Mathematical Notes], 115.

\bibitem{LandrockMichler}
P.~Landrock and G.~O. Michler.
\newblock Principal {$2$}-blocks of the simple groups of {R}ee type.
\newblock {\em Trans. Amer. Math. Soc.}, 260(1):83--111, 1980.

\bibitem{Ree}
R.~Ree.
\newblock A family of simple groups associated with the simple {L}ie algebra of
  type {$(G_{2})$}.
\newblock In {\em Proc. {S}ympos. {P}ure {M}ath., {V}ol. {VI}}, pages 111--112.
  American Mathematical Society, Providence, R.I., 1962.

\bibitem{PhDSchwabrow}
I.~Schwabrow.
\newblock {\em The centre of a block}.
\newblock PhD Thesis, University of Manchester, 2016.

\bibitem{Ward}
H.~N. Ward.
\newblock On {R}ee's series of simple groups.
\newblock {\em Trans. Amer. Math. Soc.}, 121:62--89, 1966.

\end{thebibliography}

\newpage

\begin{landscape}

\begin{table}[c]
\begin{center}
\Large
{\bf Appendix}
\end{center}

\normalsize
\caption{Character table of $^2G_2(q)$ \cite{Ward}}\label{tb:CharacterTabReeGp}
\[
\footnotesize
\begin{array}{c||cccccccccccccccc}
 & 1 & R^a\ne 1 & S^a\ne 1 & V_i & W_i & X & Y & T & T^{-1} & YT & YT^{-1} & JT & JT^{-1} & JR^a\ne J & JS^a\ne J & J\\ \hline\hline
\xi_1 & 1 &1 &1 &1 &1& 1& 1& 1& 1 & 1 & 1 & 1 & 1 & 1 & 1 & 1\\
\xi_2 & q^2-q+1 & 1 & 3 & 0 & 0 & 1-q & 1 & 1 & 1 & 1 & 1 & -1 & -1 & -1 & -1 & -1\\
\xi_3 & q^3 & 1 & -1 & -1 & -1 & 0 & 0 & 0 & 0 & 0 & 0 & 0 & 0 & 1 & -1 & q\\
\xi_4& q(q^2-q+1) & 1 & -3 & 0 & 0 & q & 0 & 0 & 0 & 0 & 0 & 0 & 0 & -1 & 1 & -q\\
\xi_5 & (q-1)m(q+1+3m)/2 & 0 & 1 & -1 & 0 & -(q+m)/2 & m & \alpha & \bar{\alpha} &  \beta & \bar{\beta} & \gamma & \bar{\gamma} & 0 & 1 & -(q-1)/2\\
\xi_6 & (q-1)m(q+1+3m)/2 & 0 & -1 & 0 & 1 & (q+m)/2 & m & \alpha & \bar{\alpha} & \beta & \bar{\beta} & -\gamma & -\bar{\gamma} & 0 & -1 & (q-1)/2\\
\xi_7 & (q-1)m(q+1+3m)/2 & 0 & 1 & -1 & 0 & -(q+m)/2 & m & \bar{\alpha} & \alpha & \bar{\beta} & \beta & \bar{\gamma} & \gamma & 0 & 1 & -(q-1)/2\\
\xi_8 & (q-1)m(q+1+3m)/2 & 0 & -1 & 0 & 1 & (q+m)/2 & m & \bar{\alpha} & \alpha & \bar{\beta} & \beta & -\bar{\gamma} & -\gamma & 0 & -1 & (q-1)/2\\
\xi_9 & m(q^2-1) & 0 & 0 & -1 & 1 & -m & -m & \delta & \bar{\delta} & \epsilon & \bar{\epsilon} & 0 & 0 & 0 & 0 & 0\\
\xi_{10} & m(q^2-1) & 0 & 0 & -1 & 1 & -m & -m & \bar{\delta} & \delta & \bar{\epsilon} & \epsilon & 0 & 0 & 0 & 0 & 0\\
\eta_r & q^3+1 & {\rm I}-6 & 0 & 0 & 0 & 1 & 1 & 1 & 1 & 1 & 1 & 1 & 1 & {\rm I}-6 & 0 & q+1\\
\eta_r' & q^3+1 & {\rm I}-6 & 0 & 0 & 0 & 1 & 1 & 1 & 1 & 1 & 1 & -1 & -1 & {\rm I}-6 & 0 & -(q+1)\\
\eta_t & (q-1)(q^2-q+1) & 0 & {\rm II}-6 & 0 & 0 & 2q-1 & -1 & -1 & -1 & -1 & -1 & -3 & -3 & 0 & {\rm II}-6 & 3(q-1)\\
\eta_t' & (q-1)(q^2-q+1) & 0 & {\rm II}-6 & 0 & 0 & 2q-1 & -1 & -1 & -1 & -1 & -1 & 1 & 1 & 0 & {\rm II}-6 & -(q-1)\\
\eta_l^- & (q^2-1)(q+1+3m) & 0 & 0 & {\rm IV}-5 & 0 & -q-1-3m & -1 & -3m-1 & -3m-1& -1 & -1 & 0 & 0 & 0 & 0 & 0\\
\eta_l^+ & (q^2-1)(q+1+3m) & 0 & 0 & 0 & {\rm IV}-5 & -q-1+3m & -1 & 3m-1 & 3m-1 & -1 & -1 & 0 & 0 & 0 & 0 & 0\\
\end{array}
\]
where $q=3^{2k+1}$, $m=3^k$, $\alpha= (-m+im^2\sqrt{3})/2$, $\beta= (-m-im\sqrt{3})/2$, $\gamma=(1-im\sqrt{3})/2$, $\delta=-m+im^2\sqrt{3}$ and $\epsilon=(m+im\sqrt{3})/2$.

\end{table}

\end{landscape}

\normalsize

\end{document}